\newcommand{\References}{references}
\newcommand{\ResourceFolder}{resources/}
\documentclass[a4paper]{article}
\usepackage[utf8]{inputenc}
\usepackage[T1]{fontenc}
\usepackage{amsmath}
\usepackage{amsfonts}
\usepackage{amssymb}
\usepackage{verbatim}
\usepackage{pdfpages}
\usepackage[colorlinks=true]{hyperref}
\usepackage{xcolor}
\usepackage{pst-all}
\usepackage{pstricks-add}
\usepackage{tabto}
\usepackage{leftidx}
\usepackage{xstring}
\usepackage{todonotes}
\usepackage{cite}
\usepackage{dsfont}
\usepackage{marvosym}
\usepackage{stmaryrd}
\usepackage{multicol}
\usepackage{graphicx}
\usepackage{chngcntr}
\usepackage{epsfig}
\usepackage{geometry}
\usepackage{tabularx}
\usepackage{mathtools}
\usepackage{mathrsfs}
\usepackage{float}
\usepackage{subcaption}
\usepackage{tikz}
\usetikzlibrary{backgrounds}
\usetikzlibrary{intersections}
\usepackage{tkz-euclide}
\usepackage{epstopdf}
\usepackage{algorithmic}
\ifpdf
  \DeclareGraphicsExtensions{.eps,.pdf,.png,.jpg}
\else
  \DeclareGraphicsExtensions{.eps}
\fi
\usepackage{pifont}

\usepackage[shortlabels]{enumitem}
\newif\ifCustomTheorems
\usepackage{amsthm}
\usepackage{cleveref}

\makeatletter
\newcommand\RedeclareMathOperator{%
  \@ifstar{\def\rmo@s{m}\rmo@redeclare}{\def\rmo@s{o}\rmo@redeclare}%
}
\newcommand\rmo@redeclare[2]{%
  \begingroup \escapechar\m@ne\xdef\@gtempa{{\string#1}}\endgroup
  \expandafter\@ifundefined\@gtempa
     {\@latex@error{\noexpand#1undefined}\@ehc}%
     \relax
  \expandafter\rmo@declmathop\rmo@s{#1}{#2}}
\newcommand\rmo@declmathop[3]{%
  \DeclareRobustCommand{#2}{\qopname\newmcodes@#1{#3}}%
}
\@onlypreamble\RedeclareMathOperator
\makeatother

\newcommand{\N}{\mathds{N}}
\newcommand{\R}{\mathds{R}}
\newcommand{\Rp}{\R_{\geq0}}
\newcommand{\Rpp}{\R_{>0}}
\newcommand{\C}{\mathds{C}}

\newcommand{\dB}{\mathds{B}}

\newcommand{\Impl}{\Longrightarrow}

\newcommand{\fa}{\ \forall \, }
\newcommand{\ex}{\ \exists \, }

\newcommand{\rbl}{\left (}
\newcommand{\rbr}{\right )}
\newcommand{\sbl}{\left [}
\newcommand{\sbr}{\right ]}

\newcommand{\nl}{\left\|}
\newcommand{\nr}{\right\|}
\newcommand{\cbl}{\left\lbrace }
\newcommand{\cbr}{\right\rbrace }

\newcommand{\Norm}[2][ ]{\nl #2 \nr_{#1}}
\newcommand{\SNorm}[1]{\Norm[\infty]{#1}}

\newcommand{\setdef}[2]{\cbl\ #1\ \left|\ \vphantom{#1} #2\ \right.\cbr}

{\left(\begin{smallmatrix}}%
{\end{smallmatrix}\right)}%

\newenvironment{smallbmatrix}%
{\left[\begin{smallmatrix}}%
{\end{smallmatrix}\right]}%

\newcommand{\GL}{\text{GL}}

\newcommand{\cC}{\mathcal{C}}
\newcommand{\cD}{\mathcal{D}}

\newcommand{\cU}{\mathcal{U}}

\newcommand{\cF}{\mathcal{F}}
\newcommand{\cG}{\mathcal{G}}
\newcommand{\cN}{\mathcal{N}}

\newcommand{\cM}{\mathcal{M}}
\newcommand{\cR}{\mathcal{R}}
\newcommand{\cY}{\mathcal{Y}}

\DeclareMathOperator*{\rf}{ref}
\DeclareMathOperator*{\esssup}{ess\,sup}

\DeclareMathOperator*{\loc}{loc}

\newcommand{\oT}{\textbf{T}}

\renewcommand{\l}{\lambda}

\newcommand{\me}{\mathrm{e}}

\newcommand{\G}{\Gamma}

\newcommand{\con}{\cC}

\RedeclareMathOperator*{\Im}{Im}
\RedeclareMathOperator*{\Re}{Re}
\renewcommand{\phi}{\varphi}

\renewcommand{\d}{\ \text{d}}
\newcommand{\ddt}{ \tfrac{\text{d}}{\text{d} t}}

\newcommand{\ve}{\varepsilon}
\newcommand{\vp}{\varphi}

\DeclareMathOperator*{\im}{im}

\makeatletter
\@ifundefined{ifCustomTheorems}{}
{
    \setcounter{section}{0}
    \counterwithout{equation}{section}
    \counterwithout{figure}{section}
    \newtheorem{definition}{Definition}[section]
    \theoremstyle{definition}
    \newtheorem{remark}[definition]{Remark}\Crefname{remark}{Remark}{Remarks}
    \newtheorem{algo}[definition]{Algorithm}\Crefname{algo}{Algorithm}{Algorithms}
    \newtheorem{example}[definition]{Example}\Crefname{example}{Example}{Examples}
    \newtheorem{assn}[definition]{Assumption}\Crefname{assn}{Assumption}{Assumptions}
    \theoremstyle{plain}
    \newtheorem{prop}[definition]{Proposition}\Crefname{prop}{Proposition}{Propositions}
    \Crefname{corollary}{Corollary}{Corollaries}
    \Crefname{assertion}{Assertion}{Assertions}
    \newtheorem{theorem}[definition]{Theorem}\Crefname{theorem}{Theorem}{Theorems}
    \newtheorem{lemma}[definition]{Lemma}\Crefname{lemma}{Lemma}{Lemmata}
}
\makeatother
\graphicspath{{\ResourceFolder}}
\newcommand{\email}[1]{\protect\href{mailto:#1}{#1}}
\newcommand\funding[1]{\protect{\bfseries Funding:} #1}
\newcommand\AMSname{AMS subject classifications}
\newenvironment{@abssec}[1]{%
     \if@twocolumn
       \section*{#1}%
     \else
       \vspace{.05in}\footnotesize
       \parindent .2in
         {\upshape\bfseries #1: }\ignorespaces 
     \fi}
     {\if@twocolumn\else\par\vspace{.1in}\fi}
\newenvironment{AMS}{\begin{@abssec}{\AMSname}}{\end{@abssec}}
\newcommand\keywordsname{Key words}
\newenvironment{keywords}{\begin{@abssec}{\keywordsname}}{\end{@abssec}}

\ifpdf
\hypersetup{
  pdftitle={Robust Funnel MPC for output tracking with prescribed performance}
  pdfauthor={Thomas Berger, Dario Dennstädt, Lukas Lanza, and Karl Worthmann}
}
\fi
\title{Robust Funnel Model Predictive Control \\ for output tracking with prescribed performance%
    \thanks{\funding{We gratefully acknowledge funding by the German Research Foundation (DFG; project number 471539468).}
}}

\author{Thomas Berger\thanks{Institut f\"ur Mathematik, Universit\"at Paderborn, Warburger Str.~100, 33098 Paderborn 
    (\email{thomas.berger@math.upd.de}, \email{dario.dennstaedt@uni-paderborn.de}).}
    \and Dario Dennstädt\footnotemark[2] \thanks{Institut f\"ur Mathematik, Technische Universit\"at Ilmenau, Weimarer Stra\ss e 25, 98693 Ilmenau, Germany 
    (\email{dario.dennstaedt@tu-ilmenau.de}, \email{lukas.lanza@tu-ilmenau.de}, \email{karl.worthmann@tu-ilmenau.de}).}
    \and Lukas Lanza\footnotemark[3]
    \and Karl Worthmann\footnotemark[3]
}

\begin{document}

\maketitle
\begin{abstract}
\noindent
We propose a novel robust Model Predictive Control (MPC) scheme 
for nonlinear multi-input multi-output systems of relative degree one with stable internal dynamics.
The proposed algorithm is a combination of funnel MPC, i.e., MPC with a particular stage cost, and the model-free adaptive funnel controller. 
The new robust funnel MPC scheme guarantees output tracking of reference signals within prescribed performance bounds --~even in the presence of unknown disturbances and a structural model-plant mismatch. 
We show initial and recursive feasibility of the proposed control scheme without imposing
terminal conditions or any requirements on the 
prediction horizon.
Moreover, we allow for model updates at runtime. To this end, we propose a proper initialization strategy, which ensures that recursive feasibility is preserved.
Finally, we validate the performance of the proposed robust MPC scheme by simulations.
\end{abstract}

\begin{keywords}
model predictive control, funnel control, nonlinear systems, reference tracking, robustness, model-plant mismatch, prescribed performance
\end{keywords}

\begin{AMS}
93B45, 93C40, 93B51, 93B52
\end{AMS}

\section{Introduction}
Model Predictive Control~(MPC) is a  well-established  control technique for linear 
and nonlinear systems  due to its ability to handle
multi-input multi-output systems under control and state constraints, see, e.g., the textbooks~\cite{GrunPann11,rawlings2017model}.
Given a model of the system, the idea is to predict the future system behavior on a finite-time horizon and,
based on the predictions, solve a respective Optimal Control Problem~(OCP). Then, the first portion
of the computed optimal control (function) is applied before
this process is repeated ad infinitum. 
Although MPC is nowadays widely used and has seen various applications, see, e.g.~\cite{QinBadg03},
there are two main obstacles: 
On the one hand, a sufficiently accurate model is required. 
On the other hand, initial and recursive feasibility have to be ensured.
The latter corresponds to 
solvability of the OCP at the successor 
time instant provided solvability of the OCP at the current time instant.
This is a non-trivial task and usually requires either some controllability properties like,
e.g., cost controllability~\cite{TunaMess06,Wort11,CoroGrun20}, in combination with a sufficiently long prediction horizon, see e.g.~\cite{boccia2014stability} and~\cite{EsteWort20} for discrete and continuous-time systems, or 
the construction of suitable terminal conditions, see e.g.~\cite{ChenAllg98,rawlings2017model}.
Especially in the presence of time-varying state or output constraints, this task becomes even more challenging as,
e.g., the extensions~\cite{AydiMuel16,KohlMuel20} to time-varying reference signals have shown. 
When considering output tracking with MPC, previous works
mostly focus on ensuring asymptotic stability of the tracking error, see, again, e.g. \cite{AydiMuel16,KohlMuel20}.
To this end, terminal sets around the reference signal and corresponding costs are introduced, resulting in so-called terminal conditions. In~\cite{KohlMuel19} tracking is achieved while avoiding such terminal constraints by assuming a suitable adaptation of cost controllability and a sufficiently long prediction horizon.

To robustly achieve output tracking,
tube-based MPC schemes construct tubes around the reference signal, 
which always contain the actual system output 
to ensure reference tracking in the presence of disturbances or uncertainties. 
For linear systems see e.g.~\cite{MaynSero05}, and for nonlinear systems see~\cite{FaluMayn14,KohlSolo21,rakovic2022homothetic} and~\cite{lopez2019dynamic}, where in the latter the tubes and the open-loop reference trajectory are simultaneously optimized  depending on its proximity to the tube boundary. 
These tubes, however, usually cannot be arbitrarily chosen 
since they have to encompass the uncertainties of the system.
To guarantee that the system output 
evolves within these tubes, terminal conditions 
are added to the optimization problem.
The tracking of a reference signal within constant bounds is studied in \cite{CairBorr16} for linear systems. 
Hereby, so-called robust control invariant sets are calculated in order to ensure
that performance,
input, and state constraints are met. The calculation of these robust control invariant sets,
however, requires a significant
computational effort and, even more important, 
the termination in finite time of the algorithm proposed in~\cite{CairBorr16} cannot be ensured.
In~\cite{YuanManz19}, the aforementioned approach was extended to systems with external disturbances.
Another approach to achieve output tracking with prescribed boundaries on the tracking error is barrier-function-based MPC, cf.~\cite{WillHeat04,Feller16, Pfitz21}.
Here, the cost function involves a term (the barrier function) which diverges, if the tracking error approaches the boundary of a given set.
However, this approach relies on imposing terminal constraints as well as terminal costs to ensure recursive feasibility.
The recently developed \textit{funnel MPC} algorithm~\cite{berger2019learningbased,BergDenn21} utilizes a particular stage cost, similar to barrier-functions, to achieve output tracking with prescribed performance.
The latter means that the tracking error evolves within (possibly time-varying) boundaries prescribed by the designer.
Given a model to be controlled, it has been rigorously proven that this novel MPC scheme is initially and recursively feasible, without imposing terminal conditions or requirements on
the prediction horizon.
However, this comes at the cost of not allowing for a-priori given state or input constraints. It is still an open question, how 
the advantages of this controller design can be maintained while incorporating such constraints.

The 
stage cost of funnel MPC is inspired by \emph{funnel control}, a high-gain adaptive feedback control law, first proposed in~\cite{IlchRyan02b}. The funnel controller is inherently robust and allows for reference tracking with prescribed performance of the tracking error for a fairly large
class of systems, see also~\cite{BergIlch21} for a comprehensive literature overview.
The idea in funnel control is that the gain is adapted based on the current distance between the error and the boundary, where the gain diverges, if the error approaches the boundary.
The boundary for the tracking error is often chosen to be a function decaying in time, which is reminiscent of a funnel.
A key feature
of funnel control is that no knowledge about the underlying system is used in the controller. 
Only structural properties of the system class such as relative degree, stable internal dynamics, and a high-gain property are assumed.
The absence of model knowledge comes at the cost that, although the input is proved to be bounded, its exact maximal value is unknown.
A relative of funnel control is \emph{prescribed performance control}~\cite{bechlioulis2008robust,bechlioulis2014low}.
Using a similar controller design, the 
tracking problem within prescribed performance boundaries is solved for a different system class. 
Since both approaches, funnel control and prescribed performance control, do not use a model of the system,
the controllers cannot ``plan ahead''.
This often results in high control values and a rapidly changing control signal with peaks.
Moreover, when implemented on real applications, both controllers require a high sampling rate to stay feasible, which may result in demanding hardware requirements.

Numerical simulations show that funnel MPC
exhibits a considerably better controller performance than pure funnel control~\cite{berger2019learningbased,BergDenn21}.
While the results in~\cite{BergDenn21} are developed for systems with relative degree one,
extensions to arbitrary relative degree based on different stage cost functions are discussed in~\cite{BergDenn22,BergDenn23b}. 
In the context of a simulation study, learning of unknown system parameters in order to apply 
funnel MPC was discussed in~\cite{berger2019learningbased}.
Beyond this, research into funnel MPC, so far, assumes the system to be precisely known and does not account
for a structural model-plant
mismatch or disturbances.
However, every model, no matter how good, deviates from the 
actual system and disturbances are omnipresent.
Furthermore, utilizing a highly detailed model is oftentimes not even desired.
Instead, one wants to use a simplified and lower-dimensional model or approximation of the system
(e.g. a discretized model for a system of partial differential equations)
to reduce complexity and computational effort, see e.g.~\cite{schilders2008model}. 
To account for external disturbances and model-plant mismatches and thereby robustify the controller, we propose \emph{robust funnel MPC}. 
This controller consists of two components: an MPC algorithm and funnel control as an additional feedback loop.
First, funnel MPC computes a control signal making use of the prediction capability of the underlying model. Then, the feedback controller (slightly) adapts the control signal using instantaneous measurement data whenever necessary to reject disturbances or to compensate a model-plant mismatch.
Therefore, the combined controller guarantees satisfaction of arbitrary output constraints.
Moreover, the proposed controller allows to update the model's state with measurement data via a \emph{proper initialization strategy}, and thereby provide data-based initial values for the optimal control problem.

\ \\
The remainder of this article is organized as follows.
In \Cref{Sec:ProblemFormulation} we define the class of systems to be controlled, and introduce the control objective.
Moreover, we discuss the components of the combined controller.
In \Cref{Sec:RobustFunnelMPC} we present the detailed controller structure, the class of models used in the model predictive controller, and the main result in \Cref{Thm:RFMPC}.
We illustrate the proposed controller with a numerical example in \Cref{sec:sim}.
\Cref{Sec:Conclusion} contains a brief conclusion of the article, and some research questions to be addressed in future work.
Most of the proofs are presented in the appendix \Cref{Sec:proofs}.

\ \\
\textbf{Nomenclature}:
In the following let $\N$ denote the natural numbers, $\N_0 = \N \cup\{0\}$, and $\R_{\ge 0}
=[0,\infty)$. By $\|x\| = \sqrt{\langle x,x\rangle}$ we denote the Euclidean norm of $x\in\R^n$. 
$\GL_n(\R)$ is the group of invertible $\R^{n\times n}$ matrices. 
 For some interval $I\subseteq\R$ and
$k\in\N$, $L^\infty(I, \R^{n})$ $\big(L^\infty_{\loc} (I, \R^{n})\big)$ is the Lebesgue space of
measurable, (locally) essentially bounded functions $f\colon I\to\R^n$ with norm $\|f \|_{\infty} =
\esssup_{t \in I} \|f(t)\|$.
$W^{k,\infty}(I,  \R^{n})$ is the Sobolev space of all functions $f:I\to\R^n$ with $k$-th order weak derivative $f^{(k)}$ and $f,f^{(1)},\ldots,f^{(k)}\in
L^\infty(I, \R^{n})$. For some $V\subseteq\R^m$ we denote by $\con^k(V,  \R^{n})$ the set of  $k$-times continuously differentiable
functions  $f:  V  \to \R^{n}$, and for brevity $\con(V,  \R^{n}) := \con^0(V,  \R^{n})$. Furthermore,  $\cR(I,\R^n)$ is the space of all 
regulated functions $f:I\to\R^n$, i.e., the left and right limits $f(t-)$ and $f(t+)$ exist for all interior points $t\in I$ and $f(a-)$ and $f(b+)$ exist whenever $a= \inf I \in I$ or $b=\sup I \in I$.
For an interval~$I$ and a function~$f:I \to \R^n$, the restriction of~$f$ to~$I$ is denoted by $f|_I$ .

\section{Problem formulation} \label{Sec:ProblemFormulation}
Before we establish the problem formulation and the control objective, we emphasize the following terminology used throughout the entire article.
The term \emph{system} refers to the actual plant to be controlled, i.e., the real system for which we do not assume availability of equations governing the dynamics.
The term \emph{model} refers to differential equations given by the control engineer. 
These model equations are used in the MPC algorithm to compute predictions.

\subsection{System class}
We consider nonlinear multi-input multi-output control systems 
\begin{equation}\label{eq:Sys}
    \begin{aligned}
        \dot{y}(t)  & = F(d(t),\oT(y)(t),u(t)),\quad y\vert_{[-\sigma,0]}=y^0\in\con([-\sigma,0],\R^m), 
    \end{aligned}
\end{equation}
with input~$u\in L^\infty_{\loc}(\Rp, \R^m)$ and output $y(t)\in\R^m$ at time $t\geq 0$.
Note that $u$ and $y$ have the same dimension~$m\in\N$.
The system consists of the \emph{unknown} nonlinear function $F\in\con(\R^p\times \R^q \times \R^m,\R^m)$,
\textit{unknown} nonlinear operator $\oT:\con([-\sigma,\infty),\R^m)\to L^\infty_{\loc}(\Rp,\R^q)$, 
and may incorporate bounded disturbances $d\in L^\infty(\Rp,\R^p)$.
Then, the constant~$\sigma\geq0$ quantifies the initial ``memory'' of the system and $y^0 \in \cC([-\sigma,0],\R^m)$ is the initial history.
The system class under consideration is characterised in detail in the following definition.
\begin{definition}[System class $\cN^{m}$] \label{Def:system-class}
     A system~\eqref{eq:Sys} belongs to the system class $\cN^{m}$, written $(d, F, \oT)
     \in\cN^{m}$, if, for some $p,q\in\N$ and $\sigma \geq0$, the following holds:
    \begin{enumerate}[label = (\roman{enumi})]
         \item  $d\in L^\infty(\Rp,\R^p)$,
        \item $\textbf{T}:\con([-\sigma,\infty),\R^m)\to L^\infty_{\loc} (\Rp, \R^{q})$ has the following properties:
        \begin{itemize}\label{Item:Operator-class}
            \item\textit{Causality}:  $\fa y_1,y_2\in\con([-\sigma,\infty),\R^m) \fa t\geq 0$:
            \[
                y_1\vert_{[-\sigma,t]} = y_2\vert_{[-\sigma,t]}
                \quad \Impl\quad
                \textbf{T}(y_1)\vert_{[0,t]}=\textbf{T}(y_2)\vert_{[0,t]}.
            \]
            \item\textit{Local Lipschitz}: 
            $\fa t \ge 0  \fa y \in \con([-\sigma,t] ; \R^m)  \ex \Delta, \delta, c > 0$ 
            $\fa y_1, y_2 \in \con([-\sigma,\infty) ; \R^m)$ with
            ${y_1|_{[-\sigma,t]} = y}$, $y_2|_{[-\sigma,t]} = y $ 
            and $\Norm{y_1(s) - y(t)} < \delta$,  $\Norm{y_2(s) - y(t)} < \delta $ for all $s \in [t,t+\Delta]$:
            \[
                \esssup_{s \in [t,t+\Delta]}  \Norm{\textbf{T}(y_1)(s) - \textbf{T}(y_2)(s) }  
                \le c \ \sup_{s \in [t,t+\Delta]} \Norm{y_1(s) - y_2(s)} .
            \]
            \item\textit{Bounded-input, bounded-output (BIBO)}:
            $\fa c_0 > 0 \ex c_1>0 \fa y \in \con([-\sigma,\infty), \R^m)$:
            \[
                \sup_{t \in [-\sigma,\infty)} \Norm{y(t)} \le c_0 \ 
                \Impl \ \sup_{t \in [0,\infty)} \Norm{\textbf{T}(y)(t)}  \le c_1.
            \]
        \end{itemize}
            \item \label{Item:high-gain-prop}
             $F\in\con(\R^p\times \R^q \times \R^m,\R^m)$ has the \emph{{perturbation} high-gain property}, i.e.
            for every compact set $K_m \subset \R^m$ there exists~$\nu\in(0,1)$ such that
            for every compact sets $K_p\subset \R^p$, $K_q\subset\R^q$
            the function
            \[
                \chi\colon\R\to\R, \
                s \mapsto \min
                \setdef{\!\!\langle v, F(\delta,z, \Delta -s v)\rangle\!\!}
                {\!\!
                    \delta\in K_p, \Delta \in K_m, z \in  K_q, v\in\R^m,~\nu \leq \|v\| \leq 1
                \!\!}
            \]
            satisfies $\sup_{s\in\R} \chi(s)=\infty$.
    \end{enumerate}
\end{definition}
The properties of the system class~$\cN^{m}$ as in \Cref{Def:system-class} guarantee that, for a control input~$u\in L^\infty_{\loc}(\Rp, \R^m)$, the system~\eqref{eq:Sys} has a solution in the sense of \textit{Carath\'{e}odory},
meaning a function~$y:[-\sigma,\omega)\to\R^m$, $\omega>0$, with $y\vert_{[-\sigma,0]}=y^0$ such that $y\vert_{[0,\omega)}$ is absolutely continuous and satisfies the functional differential equation~\eqref{eq:Sys} for almost all~$t\in[0,\omega)$.
A solution~$y$ is said to be \textit{maximal}, if it has no right extension that is also a solution.
We briefly discuss an example of a system, belonging to the class introduced above. In particular, this provides a simple candidate for
an operator~$\oT$ and illustrates the perturbation high-gain property.
\begin{example} \label{Ex:LTIsystem}
We consider a linear multi-input, multi-output system with bounded matched disturbances~$\delta$ (i.e. the disturbances act on the input channels only) of the form
\begin{align*}
    \dot x(t) &= A x(t) + B( u(t) + \delta(t)),  \quad x(0) = x^0\\
    y(t) &= C x(t),
\end{align*}
with $A \in \R^{n \times n}$ and $C, B^\top \in \R^{m \times n}$, such that $C B$ is positive definite and the zero dynamics are asymptotically stable (i.e., the system is minimum phase), that is~\cite{IlchWirt13,Isid95}
\begin{equation}\label{eq:minphaselin}
  \forall\, \lambda \in \C \text{ with } {\rm Re}\, \lambda\ge 0: \  \det \begin{bmatrix}
         \lambda I - A & B \\
        C & 0
    \end{bmatrix} \neq 0.
\end{equation}
By~\cite[Lem.~3.5]{IlchRyan07}, there exists an invertible $U \in \R^{n \times n}$ such that with $(y^\top,\eta^\top)^\top = U x$ the above system can be transformed into
\begin{align*}
    \dot y(t) & = R y(t) + S \eta(t) + \Gamma (u(t)  + \delta(t)), \\
    \dot \eta(t) &= Q \eta(t) + P y(t),
\end{align*}
where $R \in \R^{m \times m}$, $S, P^\top\in\R^{m\times (n-m)}$, $Q\in\R^{(n-m)\times (n-m)}$, $\Gamma = CB$. Furthermore, the minimum phase property~\eqref{eq:minphaselin} implies that ${\rm Re}\, \lambda < 0$ for all eigenvalues $\lambda\in\C$ of~$Q$ (the matrix~$Q$ is called \textit{Hurwitz} in this case).
Defining the linear integral operator
\begin{equation*}
    L: y(\cdot) \mapsto \left( t \mapsto \int_0^t e^{Q(t-s)} P y(s) \text{d}s \right),
\end{equation*}
and setting $d(t) := S e^{Qt} [0,I_{n-m}] U x^0$, $t\ge 0$, the system can be further rewritten as
\begin{equation*}
    \dot y(t) = d(t) + \oT(y)(t) + \Gamma (u(t)+\delta(t)) =: F(d(t),\oT(y)(t),u(t) + \delta(t)),
\end{equation*}
where $\oT : y(\cdot) \mapsto ( t \mapsto R y(t) + SL(y)(t))$.
A short calculation verifies that this operator satisfies the conditions in~\ref{Item:Operator-class} of \Cref{Def:system-class}.
In particular, the BIBO property is satisfied since~$Q$ is Hurwitz, and it is hence a consequence of the minimum phase property. Moreover, the function~$F$ has the perturbation high-gain property~\ref{Item:high-gain-prop} if, and only if, the matrix~$\Gamma $ is sign-definite, that is $v^\top \Gamma v\neq 0$ for all $v\in\R^m\setminus\{0\}$, cf.~\cite[Rem.~1.3, and Sec.~2.1.3]{BergIlch21}. 
This means that every bounded matched input disturbance~$\delta$ can be compensated.
\end{example}
 \Cref{Ex:LTIsystem} shows that minimum phase linear systems with relative degree one and sign-definite $CB$ are contained in the system class introduced in \Cref{Def:system-class}.
    Furthermore, the system class  encompasses 
    nonlinear control affine systems 
    $\dot y(t) = f(d(t),\oT(y)(t)) + g(\oT(y)(t)) u(t)$, where $f\in \con(\R^p\times\R^q,\R^m)$,~$d$ and~$\oT$ exhibit the properties as in \Cref{Def:system-class} and~$g$ is sign-definite in the sense that $v^\top g(x) v \neq 0$ for all~$x\in \R^q$ and all $v \in \R^m\setminus\{0\}$.
    We emphasize that physical effects such as \emph{backlash} or \emph{relay hysteresis}, and \emph{nonlinear time delays} can be modelled by the operator $\oT$, cf.~\cite{IlchRyan02b,BergIlch21,BergPuch20}. 
    Moreover, the operator $\oT$ in~\eqref{eq:Sys} can even be the solution operator of an infinite dimensional dynamical system, e.g. a partial differential equation.
    The latter was studied in~\cite{BergPuch22}, where a moving water tank was subject to funnel control, and the water in the tank was modelled by the linearized Saint-Venant equations.\\    
    While the first property of the operator (causality) introduced in
    \cref{Def:system-class}~\ref{Item:Operator-class} is quite intuitive, the second (locally Lipschitz) is of a more technical nature. 
    It is required to derive existence and uniqueness of solutions.
    This condition generalizes the assumption of local Lipschitz continuity of vector fields in standard state-space systems.
    The third property (BIBO) can be motivated from a practical point of view. Essentially, it is a stability condition on the ``internal dynamics'' of the system.
    The latter are represented by~$\eta$ in \Cref{Ex:LTIsystem} and, since~$Q$ is Hurwitz, $\eta$ is bounded for any bounded~$y$; hence, the $\eta$-dynamics (internal dynamics) exhibit a bounded-input, bounded-state property. 
    As shown in \Cref{Ex:LTIsystem}, this property is closely related to the BIBO property of the operator~$\oT$.
    The perturbation high-gain property introduced in \Cref{Def:system-class}~\ref{Item:high-gain-prop} is
    a modification of the so-called \emph{high-gain} property, see e.g. \cite[Def.~1.2]{BergIlch21}, and, 
    at first glance, a stronger assumption. 
    The high-gain property is essential in high-gain adaptive control and, roughly speaking, guarantees that, if a large enough input is applied, the system reacts sufficiently fast.
    For linear systems, as in \Cref{Ex:LTIsystem}, having the high-gain property implies that the system can be stabilized via high-gain output feedback, cf.~\cite[Rem.~1.3]{BergIlch21}.
    In order to account for possible bounded perturbations of the input, we require the modified property~\ref{Item:high-gain-prop}.
    It is an open problem whether the perturbation high-gain property and the high-gain property are equivalent. 
    However, in virtue of \Cref{Ex:LTIsystem} it is clear that control affine systems (linear and nonlinear) with high-gain property satisfy both properties.

\subsection{Control objective}
As a surrogate for the unknown system~\eqref{eq:Sys},  we consider a control-affine model of the form
\begin{equation}\label{eq:Mod}
    \begin{aligned}
        \dot{x}(t)  & = f(x(t)) + g(x(t)) u(t),\quad x(t_0)=x^0,\\
        y_{\rm M}(t)        & = h(x(t)),
    \end{aligned}
\end{equation}
at time $t_0\in\Rp$, with $x^0\in\R^n$, and \emph{known} functions $f\in\con^1(\R^n,\R^n)$,
$g\in\con^1(\R^n, \R^{n\times m})$, and $h\in\con^2(\R^n,\R^m)$.
Note that, in many situations, systems of the form~\eqref{eq:Mod} 
can be written in the form~\eqref{eq:Sys}.
Since the right-hand side of~\eqref{eq:Mod} is locally Lipschitz in~$x$, there exists a unique maximal solution of~\eqref{eq:Mod} for any $u\in L^\infty_{\loc}(\Rp, \R^m)$, cf.\cite[\S~10 Thm.~XX]{Walt98}.
This maximal solution is denoted by~$x(\cdot;t_0,x^0,u)$.
Contrary to the actual system~\eqref{eq:Sys}, the model~\eqref{eq:Mod} lays out its states~$x$
in an explicit way.
The model is used to make predictions about the future system output and, based on them, to compute optimal control signals.
The discrepancies between the model predictions~$y_{\rm M}(t)$ and the actual system 
output~$y(t)$ is described by the model-plant mismatch
\begin{equation*}
        e_{\rm S}(t)  := y(t)-y_{\rm M}(t).
\end{equation*}
The objective is to design a combination of a model predictive control scheme with a feedback controller (based on the measurement $y(t)$)
which, if applied to system~\eqref{eq:Sys}, allows  for reference tracking of a given 
trajectory $y_{\rf}\in W^{1,\infty}(\Rp,\R^{m})$ within predefined boundaries. To be precise, the
tracking error $t\mapsto e(t):=y(t)-y_{\rf}(t)$ shall evolve within the prescribed performance funnel
\begin{align*}
    \cF_\psi:= \setdef{(t,e)\in \Rp\times\R^{m}}{\Norm{e} < \psi(t)},
\end{align*}
see \Cref{Fig:funnel}.
\begin{figure}[H]
  \begin{center}
    \begin{tikzpicture}[scale=0.35]
    \tikzset{>=latex}
    \filldraw[color=gray!25] plot[smooth] coordinates {(0.15,4.7)(0.7,2.9)(4,0.4)(6,1.5)(9.5,0.4)(10,0.333)(10.01,0.331)(10.041,0.3) (10.041,-0.3)(10.01,-0.331)(10,-0.333)(9.5,-0.4)(6,-1.5)(4,-0.4)(0.7,-2.9)(0.15,-4.7)};
    \draw[thick] plot[smooth] coordinates {(0.15,4.7)(0.7,2.9)(4,0.4)(6,1.5)(9.5,0.4)(10,0.333)(10.01,0.331)(10.041,0.3)};
    \draw[thick] plot[smooth] coordinates {(10.041,-0.3)(10.01,-0.331)(10,-0.333)(9.5,-0.4)(6,-1.5)(4,-0.4)(0.7,-2.9)(0.15,-4.7)};
    \draw[thick,fill=lightgray] (0,0) ellipse (0.4 and 5);
    \draw[thick] (0,0) ellipse (0.1 and 0.333);
    \draw[thick,fill=gray!25] (10.041,0) ellipse (0.1 and 0.333);
    \draw[thick] plot[smooth] coordinates {(0,2)(2,1.1)(4,-0.1)(6,-0.7)(9,0.25)(10,0.15)};
    \draw[thick,->] (-2,0)--(12,0) node[right,above]{\normalsize$t$};
    \draw[thick,dashed](0,0.333)--(10,0.333);
    \draw[thick,dashed](0,-0.333)--(10,-0.333);
    \node [black] at (0,2) {\textbullet};
    \draw[->,thick](4,-3)node[right]{\normalsize$\inf\limits_{t \ge 0} \psi(t)$}--(2.5,-0.4);
    \draw[->,thick](3,3)node[right]{\normalsize$(0,e(0))$}--(0.07,2.07);
    \draw[->,thick](9,3)node[right]{\normalsize$\psi(t)$}--(7,1.4);
    \end{tikzpicture}
  \end{center}
  \vspace*{-2mm}
  \caption{Error evolution in a funnel $\mathcal F_{\psi}$ with boundary~$\psi$.} 
  \label{Fig:funnel}
\end{figure}
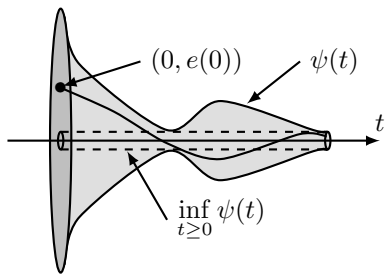
The funnel~$\mathcal{F}_{\psi}$ is determined by the choice of $\psi$
belonging to the following set of bounded functions with bounded weak derivative
\[
    \cG =\setdef{\psi\in W^{1,\infty}(\Rp,\R)}{\inf_{t \ge 0} \psi(t)>0}.
\]
The specific application usually dictates the constraints on the tracking error and thus indicates suitable choices for $\psi$. 
Note that signals evolving in~$\cF_\psi$ are not forced to asymptotically converge to~$0$.
To achieve that the tracking error~$e$ remains
within~$\cF_\psi$, it is necessary that the output~$y(t)$ of the
system~\eqref{eq:Sys} at time $t\geq 0$ is an element of 
the set
\begin{align*}
    \cD_t := \setdef
                    {y\in\R^m}
                    {\Norm{y - y_{\rf}(t)} <\psi(t)}.
\end{align*}

\subsection{Funnel MPC and funnel control}
To solve the problem of tracking a reference signal~$y_{\rf}\in W^{1,\infty}(\Rp,\R^{m})$
within pre-defined funnel boundaries~$\psi\in\cG$ for the model~\eqref{eq:Mod}
with MPC, \textit{funnel MPC} was proposed in~\cite{BergDenn21}.
Assuming  system and model to be identical, perfectly known, and of form~\eqref{eq:Mod},
the \textit{stage cost} 
$\ell:\Rp\times\R^n\times\R^{m}\to\R\cup\{\infty\}$ defined by
\begin{align}\label{eq:stageCostFunnelMPC}
    \ell(t,x,u) =
    \begin{dcases}
        \frac {\Norm{h(x)-y_{\rf}(t)}^2}{\psi(t)^2 - \Norm{h(x)-y_{\rf}(t)}^2} 
        + \l_u \Norm{u}^2,
            & \Norm{h(x)-y_{\rf}(t)} < \psi(t)\\
        \infty,&\text{else},
    \end{dcases}
\end{align}
with design parameter~$\lambda_u\in\Rp$ was proposed.
To further ensure a bounded control signal with a maximal pre-defined control value~$M>0$, 
the constraint~$\SNorm{u}\leq M$ has been added as an additional constraint to the OCP.
Using this {stage cost}, given a sufficiently large~$M>0$, 
and assuming certain structural properties which we will introduce in detail in \Cref{Sec:ModelClass},
it was shown that the following funnel MPC \Cref{Algo:FMPC} 
is initially and recursively feasible and applying this control scheme to a model of the
form~\eqref{eq:Mod} guarantees $\Norm{y_{ \rm M}(t)-y_{\rf}(t)}<\psi(t)$ for all
$t\in[0,\infty)$, provided that $\| y_{\rm M}(0) - y_{\rf}(0) \| < \psi(0)$ holds, see~\cite[Thm.~2.10]{BergDenn21}.
\begin{algo}[Funnel MPC]\label{Algo:FMPC}\ \\
    \textbf{Given:} Model~\eqref{eq:Mod}, reference signal $y_{\rf}\in
    W^{1,\infty}(\Rp,\R^{m})$, funnel function $\psi\in\cG$, control bound $M>0$,
    initial state $x^0$ such that $h(x^0) \in\cD_{0}$, and stage cost function~$\ell$ as in~\eqref{eq:stageCostFunnelMPC}.\\
    \textbf{Set} time shift $\delta >0$, prediction horizon $T\geq\delta$, define the time sequence~$(t_k)_{k\in\N_0} $ by $t_k := k\delta$
    and set the current index~$k=0$.\\
    \textbf{Steps:}
    \begin{enumerate}[label=(\alph*), ref=\alph*, leftmargin=*]
    \item\label{agostep:FMPCFirst} Obtain a measurement of the state $x$ of~\eqref{eq:Mod}
    at time~$t_k$ and set $\hat x :=x(t_k)$.
    \item Compute a solution $u^{\star}\in L^\infty([t_k,t_k +T],\R^{m})$ of
    \begin{equation*}
            \mathop
            {\operatorname{minimize}}_{\substack{u\in L^{\infty}([t_k,t_k+T],\R^{m}),\\\SNorm{u}\leq M}}  \quad
            \int_{t_k}^{t_k+T}\ell(t,x(t;t_k,\hat{x},u),u(t))\d t .
    \end{equation*}
    \item Apply the time-varying control
        \[
            \mu:[t_k,t_{k+1})\times\R^n\to\R^m, \quad \mu(t,\hat x) =u^{\star}(t)
        \]
        to the model~\eqref{eq:Mod}.
        Increment~$k$ by~$1$ and go to Step~\eqref{agostep:FMPCFirst}.
    \end{enumerate}
\end{algo}

\begin{remark}
    The cost function~\eqref{eq:stageCostFunnelMPC} used in the funnel MPC~\Cref{Algo:FMPC} is inspired by the~\emph{funnel controller}. 
    For systems~\eqref{eq:Sys} with $(d,F,\oT) \in \cN^m$ as in~\Cref{Def:system-class} 
    and given reference trajectory $y_{\rm ref} \in W^{1,\infty}(\Rp,\R^m)$ it was shown in \cite[Thm.~1.9]{BergIlch21} 
    that the tracking error $e(t):=y(t)-y_{\rf}(t)$ always evolves within the performance funnel~$\cF_\psi$ by applying the control signal
    \begin{equation} \label{eq:u_fc}
        u(t) = (N \circ \alpha)(\| e(t)\slash \psi(t)\|^2) e(t)\slash \psi(t),
    \end{equation}
    where $N \in \con(\Rp,\R)$ is a surjection and $\alpha \in\con([0,1) ,[1,\infty))$ is a bijection. A simple (and often used) feasible choice is $\alpha(s) = 1/(1-s)$ and $N(s) = s \sin(s)$.    
\end{remark}

\section{Robust funnel MPC} \label{Sec:RobustFunnelMPC}

We present in detail the idea of how to combine funnel MPC~\Cref{Algo:FMPC},
see also~\cite{berger2019learningbased, BergDenn21},
with results on the \emph{model-free} funnel controller 
to achieve the control objective in the presence 
of a mismatch between the system~\eqref{eq:Sys} and the model~\eqref{eq:Mod}.
The idea is depicted in \Cref{fig:ControlStructure}.
\begin{figure}[ht]
    \centering
     \scalebox{.94}{
    \begin{tikzpicture}[very thick,%
        scale=0.58,%
        node distance = 9ex,
        box/.style={fill=white,rectangle, draw=black},
        blackdot/.style={inner sep = 0, minimum size=3pt,shape=circle,fill,draw=black},%
        blackdotsmall/.style={inner sep = 0, minimum size=0.1pt,shape=circle,fill,draw=black},%
        plus/.style={fill=white,circle,inner sep = 0,very thick,draw},%
        metabox/.style={inner sep = 3ex,rectangle,draw,dotted,fill=gray!20!white}]
        \begin{scope}[scale=0.5]
            \node (sys) [box,minimum size=9ex,xshift=-1ex]  {System};
            \node(FC) [box, below of = sys,yshift=-8ex,minimum size=9ex] {Funnel controller};
            \node(fork1) [plus, right of = FC, xshift=18ex] {$+$};
            \node(fork9) [blackdot, inner sep = 0pt, right of = fork1, xshift=-2ex ] {};
            \node(fork2) [plus, left of = FC, xshift=-15ex] {$+$};
            \node(fork3) [blackdot, left of = fork2, xshift=-0ex] {};
            \node(MPC) [box, left of = fork3,xshift=-8ex,minimum size=9ex] {FMPC};
            \node(MPCin1) [minimum size=0pt, inner sep = 0pt, below of = MPC, yshift=4.5ex, xshift=2ex] {};
            \node(MPCin1Desc) [minimum size=0pt, inner sep = 0pt, below of = MPCin1, yshift=5ex, xshift=2.5ex] {$y$};
            \node(MPCin2) [minimum size=0pt, inner sep = 0pt, below of = MPC, yshift=4.5ex, xshift=-2ex]{};
            \node(MPCin2Desc) [minimum size=0pt, inner sep = 0pt, below of = MPCin2, yshift=5ex, xshift=-2.5ex] {$y_{\rf}$};
            \node(refin) [minimum size=0pt, inner sep = 0pt, below of = MPC, yshift=-2ex, xshift=-2ex] {};
            \node(Mod) [box, above of = MPC,yshift=8ex,minimum size=9ex] {Model};
            \node(fork4) [blackdot, left of = MPC, xshift=-5ex] {};
            \node(fork5) [minimum size=0pt, inner sep = 0pt, below of = fork4, yshift=-5ex] {};
            \node(fork6) [minimum size=0pt, inner sep = 0pt, below of = fork9, yshift=-7ex] {};
            \node(fork7) [minimum size=0pt, inner sep = 0pt, below of = MPCin1, yshift=-2.5ex] {};
            \draw[->] (refin) -- (MPCin2) node[pos=0.4,left] {};
            \draw[->] (MPC) -- (fork2) node[pos=0.2,above] {$u_{\rm FMPC}$};
            \draw[->] (fork3) |- (Mod);
            \draw (Mod) -| (fork4) node[pos=0.3,above] {$y_{\rm M}$};
            \draw[->] (fork4) -- (MPC);
            \draw[-] (sys) -| (fork9)node[pos=0.05,above] {$y$} ;
            \draw[->] (fork9) -- (fork1) node[pos=0.6,right, above] {$+$};
            \draw[-] (fork9) -- (fork6.south);
            \draw[-] (fork6.east) -- (fork7.west);

            \path[->,name path=line1] (fork7.south) -- (MPCin1){};
            \draw[->,name path=line2] (fork5.west) -| (fork1) node[pos=0.9,left] {$-$};
            \path [name intersections={of = line1 and line2}];
            \coordinate (S)  at (intersection-1);          
            \path[name path=circle] (S) circle(5.mm);
            \path [name intersections={of = circle and line1}];
            \coordinate (I1)  at (intersection-1);
            \coordinate (I2)  at (intersection-2);
            \tkzDrawArc[color=black, very thick](S,I1)(I2);
            \draw[-] (fork7.south) |- (I2);
            \draw[->] (I1) -- (MPCin1);

            \draw[->] (fork1) -- (FC) node[midway,above] {$e_{\rm S}=y - y_{\rm M}$};
            \draw[->] (FC) -- (fork2) node[pos=0.4,above] {$u_{\rm FC}$};
            \draw[->] (fork2) |- (sys) node[pos=0.73,above] {$u=u_{\rm FMPC} + u_{\rm FC}$};
            \draw (fork4) -- (fork5.south);
        \end{scope}
        \begin{pgfonlayer}{background}
            \fill[red!20] (-18,-7.5) rectangle (-8,2.);
            \fill[blue!20] (-5.,-7.5) rectangle (8,-2.);
            \node at (-12.9,-9.5) {\color{red}{\large Model-based controller component}};
            \node at (1.5,-9.5) {\color{blue}{\large Model-free controller component}};
        \end{pgfonlayer}
    \end{tikzpicture}
    } 
    \caption{Structure of the robust funnel MPC scheme}
    \label{fig:ControlStructure}
\end{figure}
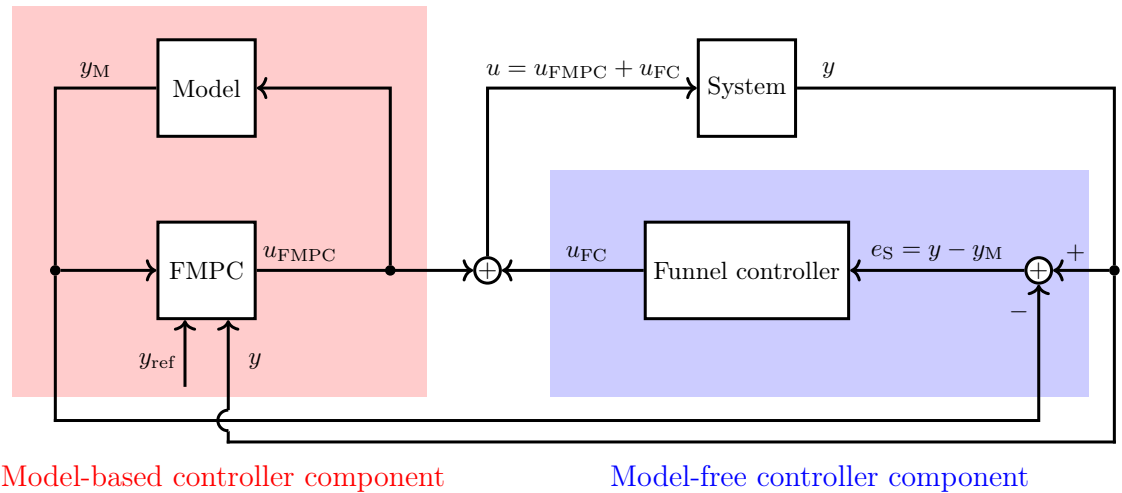
The left block with red background contains the model~\eqref{eq:Mod}, the funnel MPC \Cref{Algo:FMPC} and a given reference trajectory $y_{\rm ref}$. 
We emphasize that the model is given by the designer and hence it is known.
Funnel MPC achieves, for given $\psi \in \cG$, that the model's output $y_{\rm M}$ tracks the reference with predefined accuracy,
i.e., $\|e_{\rm M}(t)\| = \|y_{\rm M}(t) - y_{\rm ref}(t) \| < \psi(t)$ for all $t \ge 0$,
while the control input $u_{\rm FMPC}$ minimizes the stage cost~\eqref{eq:stageCostFunnelMPC}, cf.~\cite{BergDenn21}.

The right block contains the system to be controlled, and a funnel control feedback loop. 
Given a reference signal $\rho \in W^{1,\infty}(\Rp,\R^m)$ and a funnel function $\vp \in \cG$, funnel control~\eqref{eq:u_fc} achieves that the system's output $y$ tracks the reference with predefined accuracy,
i.e., $\| y(t) - \rho(t) \| < \vp(t)$ for all $t \ge 0$, cf. \cite{IlchRyan02b,BergLe18a,BergIlch21}.
Note that this control scheme is model free.
In other words, the control~\eqref{eq:u_fc} can be successfully applied to any system of the form~\eqref{eq:Sys} satisfying some structural properties without knowledge of the system's parameters.
Merely the  initial error is required to be within the funnel boundaries.  

The advantage of funnel MPC is that the control input $u_{\rm FMPC}$ is optimal in the sense of minimizing a given cost function. The advantage of funnel control is that it does not require the knowledge of a model and is hence inherently robust. We admit that by combining both control strategies we loose both:
the combination is neither model-free, nor is the control signal optimal.
The main idea is to \emph{robustify} the funnel MPC scheme w.r.t. uncertainties and disturbances, which is indeed guaranteed by the funnel controller component.
However, the funnel controller should remain inactive as long as the prediction~$y_{\rm M}$ by the model of the system output~$y$ is {sufficiently} 
accurate. 
It should only be active when the system is, according to the model, in a critical state, meaning that the predicted error $e_{\rm M}$ is close to the funnel boundary $\psi$.
In this case the funnel controller achieves that system and model behave similarly.
Therefore, $\vp = \psi - \|e_{\rm M}\|$ is chosen as funnel function for the funnel controller.
This approach ensures that the funnel controller only compensates disturbances when necessary and hence the combined control signal $u = u_{\rm FMPC} + u_{\rm FC}$ deviates from the optimal control $u_{\rm FMPC}$ as slightly as possible.

Let us give a more precise description of the controller structure depicted in \Cref{fig:ControlStructure}:
On the left hand side (red box), funnel MPC 
computes the control signal $u_{\rm FMPC}(t)$, $t \in [t_k,t_k+\delta)$, and the corresponding output is $y_{\rm M}(t)$, $t \in [t_k,t_k+\delta)$, which is handed over to the funnel controller 
on the right side of \Cref{fig:ControlStructure} (blue box) and serves as a reference signal for system~\eqref{eq:Sys}.
Via the application of the funnel control $u_{\rm FC}$ the system's output $y$ follows the model's output $y_{\rm M}$ with pre-defined accuracy, i.e., 
$\|e_{\rm S}(t)\| = \| y(t) - y_{\rm M}(t) \| < \vp(t)$, where $\vp = \psi - \|e_{\rm M}\|$ as mentioned above.
The control signal $u = u_{\rm FMPC} + u_{\rm FC}$ is applied to the system, which has the following consequence:
If the model and the system coincide and are initialized equally, then the application of the control $u_{\rm FMPC}$ has the same effect on both dynamics, and so the system's output $y$ equals the model's output $y_{\rm M}$, i.e., $\| y(t) - y_{\rm ref}(t) \| = \| y_{\rm M}(t) - y_{\rm ref}(t) \| < \psi(t)$.
Invoking~\eqref{eq:u_fc}, this in particular means $u_{\rm FC} = 0$.
If, however, the model does not match the system, then $e_S(t) \neq 0$ and $u_{\rm FC}(t) \neq 0$.
Roughly speaking, the more model and system differ, the more the funnel controller has to compensate; and the better the model matches the system, the more the control $u_{\rm FMPC}$ can contribute to the tracking task.

Since $u_{\rm FMPC}$ is a (piecewise) optimal control calculated using the model~\eqref{eq:Mod},
the aim is to keep $u_{\rm FC}$ as small as possible while achieving the tracking objective.
When the OCP is solved in funnel MPC at time instance~$t_{k}$,
it is necessary to update the initial value~$x(t_k)$ of the model~\eqref{eq:Mod},
if there is a mismatch between $y(t_{k})$ and $y_{\rm M}(t_{k})$.
One has to find a  \emph{proper initialization} for the model, meaning, based on the information of $y(t_{k})$, 
it is necessary to find a starting configuration of the model such that its output $y_{\rm M}(t_{k})$ is close to the system's output $y(t_{k})$
in order to calculate a control signal~$u_{\rm FMPC}$ for the next time interval~$[t_{k},t_{k+1}]$, 
which contributes to the tracking task.
Otherwise, due to deviation between system and model, it might be possible, that the control signal~$u_{\rm FMPC}$
is unsuitable and needs to be compensated by~$u_{\rm FC}$.

The remainder of this section is organized as follows. In \Cref{Sec:ModelClass} we introduce the class of models to be used in the robust funnel MPC \Cref{Algo:RobustFMPC}. 
Then in \Cref{Sec:ControllerStructure} we discuss in detail the controller's structure. To prove recursive feasibility of the proposed MPC algorithm, we introduce a \emph{proper initialization strategy} in \Cref{Def:Initialization}.
In order to avoid that the funnel feedback controller unnecessarily compensates small model-plant mismatches, we introduce an \emph{activation function}.
With the introductory work at hand, we finally establish the main result in \Cref{Sec:MainResult}.

\subsection{Model class} \label{Sec:ModelClass}
We stress that the model~\eqref{eq:Mod} itself is, in its essence,
a controller design parameter~-- the better the model, the better the controller performance.
But we will be able to show that even with a very poor model the robust funnel MPC \Cref{Algo:RobustFMPC} achieves the control objective.
Since in the later analysis we utilize the so-called \emph{Byrnes-Isidori form},
we make the following assumptions about the model~\eqref{eq:Mod} throughout this work.

\begin{assn}\label{ass:BIF}
    The model~\eqref{eq:Mod} has global relative degree $r=1$, i.e., the \emph{high-gain matrix} $\Gamma(x) := ( h' g)(x)$ is invertible for all $x\in\R^n$, where $h'$ denotes the Jacobian of~$h$. 
    Additionally, $h^{-1}(0)$ is diffeomorphic to $\R^{n-m}$ and the mapping $x\mapsto G(x) := \im g(x)$ is involutive, i.e., 
    for all smooth vector fields $V_i: \R^n\to\R^n$ with $V_i(x)\in G(x)$, $i \in \{1,2\}$, we have that the Lie bracket $[V_1, V_2](x) = V_1'(x) V_2(x) - V_2'(x) V_1(x)$ satisfies $[V_1, V_2](x)\in  G(x)$ for all $x\in\R^n$. 
\end{assn}
If the model fulfills \Cref{ass:BIF}, then, by~\cite[Cor.~5.7]{ByrnIsid91a} there exists a diffeomorphism $\Phi:\R^n\to\R^n$ such that the coordinate transformation $(y_{\rm M}(t),\eta(t)) = \Phi(x(t))$ puts the model~\eqref{eq:Mod} into Byrnes-Isidori form
\begin{subequations}\label{eq:BIF}
    \begin{align}
        \dot y_{\rm M}(t) &= p\rbl y_{\rm M}(t),\eta(t)\rbr + \G\rbl \Phi^{-1}\rbl y_{\rm M}(t),\eta(t)\rbr\rbr\,u(t),\quad\ (y_{\rm M}(0),\eta(0)) = (y_{\rm M}^0,\eta^0) = \Phi(x^0), \label{eq:output_dyn}\\
        \dot \eta(t) &= q\rbl y_{\rm M}(t),\eta(t)\rbr,\label{eq:zero_dyn}
    \end{align}
\end{subequations}
where $p\in \con^1(\R^m\times\R^{n-m},\R^m)$ and $q\in \con^1(\R^m\times\R^{n-m},\R^{n-m})$. Here, equation~\eqref{eq:zero_dyn} describes the so-called \emph{internal dynamics}.

\begin{remark}
For specific applications, if possible, the model should be chosen such that it is already in Byrnes-Isidori form~\eqref{eq:BIF}. The reason is that finding the diffeomorphism $\Phi : \R^n \to \R^n$ is a hard task in general, cf. \cite{Isid95}. In \cite{Lanz21} an approach is presented to compute $\Phi$ algorithmically.
    In the simple but relevant case of linear output $h(x) = H x$, $H \in \R^{m \times n}$, and constant input distribution $g(x) = G \in \R^{n \times m}$ such that $HG$ is invertible, the assumptions of~\cite[Cor.~5.7]{ByrnIsid91a} are satisfied and following the derivations in~\cite{Lanz21}, the transformation can be written as
    \begin{equation}\label{eq:Phi-lin}
        \begin{pmatrix} y \\ \eta \end{pmatrix} = \Phi (x) 
        = \begin{bmatrix}
        H \\
            V^\dagger( I_n - G (HG)^{-1} H)
        \end{bmatrix}  x , \quad V \in \R^{n \times (n-m)} \ \text{with} \ \im V = \ker H,
    \end{equation}
    where $V^\dagger \in \R^{(n-m) \times n}$ denotes the pseudoinverse of $V$.
    In this particular case the inverse transformation is given by $x = \Phi^{-1}(y,\eta) = G (HG)^{-1} y + V \eta$. 
    Then equations~\eqref{eq:BIF} read
    \begin{align*}
        \dot y_{\rm M}(t) &= Hf\big( G (HG)^{-1} y_{\rm M}(t) + V \eta(t)\big) + HG u(t), \\
        \dot \eta(t) &= V^\dagger( I_n - G (HG)^{-1} H)  f(G (HG)^{-1} y_{\rm M}(t) + V \eta(t) ).
    \end{align*}
\end{remark}

With the assumption made above, we may now introduce the class of models to be considered. 
\begin{definition}[Model class~$\cM^m$] \label{Def:Model_Class}
    A model~\eqref{eq:Mod} belongs to the model class $\cM^{m}$, written $(f, g, h) \in\cM^{m}$, if it satisfies \Cref{ass:BIF} and 
    the internal dynamics~\eqref{eq:zero_dyn} satisfy the following \emph{bounded-input, bounded-state} (BIBS) condition:
    \begin{equation}\label{eq:BIBO-ID}
        \fa c_0 >0  \ex c_1 >0  \fa  \eta^0\in\R^{n-m}
       \fa  \zeta\in L^\infty(\Rp,\R^{m}):\ \Norm{\eta^0}+ \SNorm{\zeta}  \leq c_0 \!\implies\! \SNorm{\eta (\cdot;0,\eta^0,\zeta)} \leq c_1,
    \end{equation}
    where $\eta (\cdot;0, \eta^0,\zeta):\Rp\to\R^{n-m}$ denotes the unique global solution of~\eqref{eq:zero_dyn} when $y_{\rm M}$ is substituted by~$\zeta$.
\end{definition}
Note that the BIBS assumption \eqref{eq:BIBO-ID} ensures that the maximal solution $\eta (\cdot;0, \eta^0,\zeta)$ can in fact be extended to a global solution.
Since in many applications the usage of a linear model is reasonable, we emphasize that linear systems of the form
\begin{align*}
    \dot x(t) &= A x(t) + B u(t), \\
    y_{\rm M}(t) &= C x(t),
\end{align*}
with $A \in \R^{n \times n}$ and $C, B^\top \in \R^{m \times n}$ belong to $\cM^m$, provided that $C B$ is invertible and~\eqref{eq:minphaselin} is satisfied; this follows from \Cref{Ex:LTIsystem}.
\begin{remark}
    Although there are many systems belonging to both the model class~$\cM^m$ and the system class~$\cN^m$,
    neither the set of admissible models~$\cM^m$ is a subset of all considered systems~$\cN^m$ nor the opposite is true. 
    Every system  $(d, F, \oT)\in\cN^m$ which involves a time delay, e.g. $\oT(y)(t) = y(t-\sigma)$ with $\sigma > 0$, cannot belong to~$\cM^m$.
    On the other hand, for an example of a model belonging to $\cM^m$ but not to $\cN^m$ consider $(f,g,h)\in\cM^m$ with high-gain matrix 
    $\Gamma = h'g = \begin{smallbmatrix} -1 & 0\\ 0 & 1\end{smallbmatrix}$, which is invertible but not sign-definite.
    At first glance, it might seem advantageous to also require the high-gain matrix of a model to be sign-definite, in order to ensure a closer approximation of the real system. However, if the model is automatically generated by means of a learning algorithm using system data, see~e.g.~\cite{lanza2023safe},
    this requirement might unnecessarily restrict the set of admissible models and limit the capabilities of the used learning process.
\end{remark}

\subsection{Controller structure} \label{Sec:ControllerStructure}
Before we establish the main result, we informally introduce some aspects of the robust funnel MPC algorithm.

\ \\
\textbf{Proper initialization strategy.}
While the model~\eqref{eq:Mod} lays out its system states in an explicit way, the
internal states of the system~\eqref{eq:Sys} are unknown.
Moreover, only the measurement of the system output~$y$ is available. 
However, the measurement of the current state of the model in Step~\eqref{agostep:FMPCFirst}
of \Cref{Algo:FMPC} (funnel MPC) is essential for its functioning. 
When applying {the control resulting from} MPC to 
system~\eqref{eq:Sys}, it is therefore necessary to 
initialize the current model state~$\hat{x}$ based on the measured system output~$\hat{y}$
and the previous prediction of the model state~$x^{\rm pre}=x(t_{k+1};t_k,\hat{x}_k,u_{\rm FMPC})$ in a more sophisticated way.
There are two reasonable possibilities to initialize~$\hat{x}$. 
One option is to choose the model output such that it coincides with the system output, i.e., 
\begin{equation}\label{eq:SetModelToSystemData}
    h(\hat{x}) = \hat{y}.
\end{equation}
However, since we cannot measure the internal state of the system, there are two ways to treat the internal dynamics~\eqref{eq:zero_dyn} of the model.
Either, we set $\hat x$ such that we do not manipulate the internal dynamics, i.e.,
\[
    \sbl 0, I_{n-m}\sbr\Phi(\hat{x})=\sbl 0, I_{n-m}\sbr\Phi(x^{\rm pre}),
\]
or we ``reset'' the internal dynamics of the model, i.e., for a fixed a-priori defined bound~$\xi\in\Rp$ for the internal dynamics, we initialize $\hat x$ such that 
\[
    \| [0,I_{n-m}]\Phi(\hat x) \| \le \xi .
\]
Using the diffeomorphism~$\Phi$ and the Byrnes-Isidori form~\eqref{eq:BIF},
it is in fact always possible to find a model state which satisfies one of these properties and which coincides with the system output as in~\eqref{eq:SetModelToSystemData},
as the following lemma shows.
\begin{lemma}\label{Lemma:NonTrivialInit}
    Let a model $(f,g,h)\in \cM^m$ as in~\Cref{Def:Model_Class} be given and let $\Phi$ be the diffeomorphism from \Cref{ass:BIF}.
    For $\xi \in \Rp$, $x^{\rm pre}\in\R^n$, and $\hat{y} \in\R^m$ the set 
    \begin{equation}\label{eq:SetOmegaNonEmpty}
        \tilde{\Omega}_{\xi}(x^{\rm pre},\hat{y}):=\setdef{ x \in \R^n }{
        \begin{array}{l}
            h(x)=\hat{y},\\
            \sbl 0, I_{n-m}\sbr\Phi(x)=\sbl 0, I_{n-m}\sbr\Phi(x^{\rm pre}) 
            \ \text{or} \ \| [0,I_{n-m}]\Phi(x) \| \le \xi
        \end{array} }
    \end{equation}
    is non-empty. 
\end{lemma}
\begin{proof}
    Let $z:=\Phi^{-1}(\hat{y},0_{n-m})$. Then, recalling $[I_m,0]\Phi(\cdot) = h(\cdot)$,
    we have $h(z)=\hat{y}$. 
    Further, $\Norm{[0,I_{n-m}]\Phi(z)}=\Norm{[0,I_{n-m}]\Phi(\Phi^{-1}(\hat{y},0_{n-m}))} = 0 \leq\xi $.\
    Thus, $z$ is an element of $\tilde{\Omega}_{\xi}(x^{\rm pre},\hat{y})$. 
\end{proof}
Without choosing the model output such that it coincides with the system output, i.e. without fulfilling \eqref{eq:SetModelToSystemData}, 
another option to initialize the model is to allow for a (temporary) open-loop operation of~\Cref{Algo:FMPC}, meaning that we allow initializing 
the current model state~$\hat{x}$ with the previous prediction~$x^{\rm pre}$.
This allows to initialize the model without computing the diffeomorphism~$\Phi$ and the Byrnes-Isidori form~\eqref{eq:BIF} (if it was not possible to choose a model in this form).
The following definition formalizes these possibilities.
\begin{definition}\label{Def:Initialization}
    Let a model $(f,g,h)\in \cM^m$ as in~\Cref{Def:Model_Class} be given and let $\Phi$ be the diffeomorphism from \Cref{ass:BIF}.
    Let $\xi \in \Rp$, $x^{\rm pre}\in\R^n$, and $\hat{y} \in\R^m$. Using~\eqref{eq:SetOmegaNonEmpty}  define the set
    \[
        \Omega_{\xi}(x^{\rm pre}, \hat{y}):=\tilde{\Omega}_{\xi}(x^{\rm pre},\hat{y})
        \cup
        \cbl
        x^{\rm pre}
        \cbr.
    \]
    We call $\hat{x}\in\Omega_{\xi}(x^{\rm pre},\hat{y})$ a \emph{proper initialization} and 
    a function $\kappa_{\xi}:\R^n\times\R^m\to\R^n$ with $\kappa_{\xi}(x,y)\in \Omega_{\xi}(x,y)$ for all $(x,y)\in\R^n\times\R^m$ a \emph{proper initialization strategy}.
\end{definition}
We emphasize that, for $x^{\rm pre}\in\R^n$ and~$\hat{y}\in\R^m$, 
there always exists a proper initialization, not only because of $x^{\rm pre}\in\Omega_{\xi}(x^{\rm pre},\hat{y})$, but also according to~\Cref{Lemma:NonTrivialInit}.\\

\textbf{Funnel boundary and activation function.}
In \Cref{Algo:RobustFMPC} (robust funnel MPC) we choose the funnel function for the funnel controller very specifically. Namely, we use $\vp = \psi - \|e_{\rm M}\|$, where $e_{\rm M} = y_{\rm M} - y_{\rm ref}$.
This choice reflects the following idea:
If the error $e_{\rm M}$ is small, then the funnel boundary $\vp$ is approximately given by the MPC funnel boundary $\psi$.
If, however, the error $e_{\rm M}$ is close to $\psi$, then $\vp$ becomes tight, such that the system's
output is forced to be very close to the model's output.
This means, whenever the tracking is critical, the system is forced to behave very similar to the model such that even in critical situations it is reasonable to use  \emph{model} predictive control.
The choice of $\vp$ ensures that the tracking error evolves within the funnel $\psi$, i.e., we have
\[      
\forall \, t \ge 0 \, : \  \| y(t) - y_{\rm ref}(t) \| < \psi(t). 
\]

\noindent
Besides the particular choice of $\vp$ we will utilize an \emph{activation function}, i.e. a continuous function $\beta : [0,1] \to [0,\beta^+]$, $\beta^+ > 0$,
in order to activate or deactivate the funnel control signal $u_{\rm FC}$, depending on the magnitude of the error $e_{\rm M}$. 
Compared to the funnel controller without activation function, the incorporation of the activation function~$\beta(\cdot)$ in the control law~\eqref{alg:eq:FC} results in a scaling of the gain $(N \circ \alpha)(\cdot)$.
However, since adaption of the gain is not affected by~$\beta(\cdot)$, the resulting control~$u_{\rm FC}$ is as large as required. 
Using an activation function in the funnel control law is justified by the following result.
\begin{prop} \label{Prop:FC_dist}
Consider a system~\eqref{eq:Sys} with $(d,F,\oT) \in \cN^m$ as in~\Cref{Def:system-class}.
Let $y^0\in \cC([-\sigma,0],\R^m)$, $\sigma\ge 0$, $D \in  L^\infty(\Rp,\R^m)$, $\beta\in\con([0,1],[0,\beta^+])$ be an activation function with $\beta^+>0$, $\rho \in W^{1,\infty}(\Rp, \R^m)$ and $\vp \in \cG$ be given such that $\|y^0(0) - \rho(0)\| < \vp(0)$.
Then the application of
\begin{equation*}
    u(t) =  \beta( \| e(t)\slash\vp(t) \| ) (N \circ \alpha) (\| e(t)\slash\vp(t) \|^2 )  e(t)\slash\vp(t), \quad e(t) := y(t) - \rho(t),
\end{equation*}
to the system 
\begin{equation*}
    \dot y(t) =  F\big(d(t),\oT(y)(t), D(t)+u(t)\big), \quad y|_{[-\sigma,0]}=y^0,
\end{equation*}
yields a closed-loop initial value problem, which has a solution, every solution can be maximally extended, and every maximal solution $y: [0,\omega) \to \R^m$ has the following properties
\begin{enumerate}[label = (\roman{enumi})]
    \item the solution is global, i.e., $\omega = \infty$, \label{assertion_omega_infty}
    \item all signals are bounded, in particular, $u, \dot y \in L^\infty(\Rp,\R^m)$ and $y \in L^{\infty}([-\sigma,\infty),\R^m)$, \label{assertion_bounded_signals}
    \item the tracking error evolves within prescribed error bounds, i.e., \label{assertion_funnel_property}
    \begin{equation*}
        \forall \, t \ge 0 \, : \  \| y(t) - \rho(t) \| <\vp(t).
    \end{equation*}
\end{enumerate}
\end{prop}
The proof is relegated to \Cref{Sec:proofs}.

Since very small deviations between $y(t)$ and $y_{\rm M}(t)$ can be neglected, we use the term  $\beta(\| e_{\rm S}(t)\slash\vp(t)\|)$, where $e_{\rm S} = y - y_{\rm M}$, which can be set to zero when $e_{\rm S}$ is small.
A reasonable and simple choice would be
    \begin{equation*}
        \beta(s) = \begin{dcases}
            0, & s \le S_{\rm crit}, \\
            s-S_{\rm crit}, & s \ge S_{\rm crit},
        \end{dcases}
    \end{equation*}
for $S_{\rm crit} \in (0,1)$. In this particular case we may set $\beta^+ = 1-S_{\rm crit}$.
In the context of machine learning, in particular, artificial neural networks, this type of functions is known as \emph{rectified linear unit} (ReLU).
Note that $\beta$ defined above satisfies $\beta(S_{\rm crit}) = 0$, whereby it is a continuous function and thus the funnel controller contributes continuously to the overall control signal. 

\noindent
\textbf{Robust funnel MPC algorithm.}
With the definitions and concepts introduced so far at hand, we are in the position to establish the \emph{robust funnel MPC algorithm}.
\begin{algo}[Robust funnel MPC]\label{Algo:RobustFMPC}\ \\
    \textbf{Given:}
    \begin{itemize}
        \item instantaneous measurements of the output $y(t)$ of system~\eqref{eq:Sys},
            reference signal $y_{\rf}\in W^{1,\infty}(\Rp,\R^{m})$,
            funnel function $\psi\in\cG$,
        \item model~\eqref{eq:Mod} with $(f,g,h)\in\cM^m$ as in~\Cref{Def:Model_Class} and diffeomorphism~$\Phi$ as in \Cref{ass:BIF},
            stage cost function~$\ell$ as in~\eqref{eq:stageCostFunnelMPC},
            $\xi\in\Rp$,
            a proper initialization strategy~$\kappa_{\xi}:\R^n\times\R^m\to\R^n$ as in~\Cref{Def:Initialization},
            bound $M>0$ for the MPC control signal, and the initial value $x^0$ of the model's state such that
    \begin{equation} \label{eq:ModelInitalValues}
            x^0\in X^0:=
            \setdef{x\in\R^n}
            {\begin{array}{l}
                \Norm{y(0) - h(x)}<\psi(0)-\Norm{h(x)-y_{\rf}(0)},\\
                \Norm{[0,I_{n-m}]\Phi(x)} \le \xi
            \end{array} },
    \end{equation}
        \item a surjection $N \in \con(\Rp,\R)$,
        a bijection $\alpha \in\con([0,1), [1,\infty))$, and an activation function~$\beta\in\con([0,1],[0,\beta^{+}])$ with $\beta^{+}>0$.
    \end{itemize} 
    \textbf{Set} time shift $\delta >0$, prediction horizon $T\geq\delta$, and index $k:=0$.\\
    \textbf{Define} the time sequence~$(t_k)_{k\in\N_0} $ by $t_k := k\delta$ and the first element
    of the sequence of  predicted states~$(x^{\rm pre}_k)_{k\in\N_0}$  by $x^{\rm pre}_0:=x^0$.\\
    \textbf{Steps:}
    \begin{enumerate}[label=(\alph{enumi}), ref=(\alph{enumi}), leftmargin=*]
    \item \label{agostep:RobustFMPCFirst} 
    Obtain a measurement $y(t_k) =: \hat y_k$ of the system output~$y$ at time~$t_k$, and 
    choose a \emph{proper initialization} $\hat x_k = \kappa_{\xi}(x^{\rm pre}_{k},\hat y_k)
    \in \Omega_{\xi}(x^{\rm pre}_k,\hat{y}_k)$ for the model. 
    \item Compute a solution $u_{\rm FMPC}\in L^\infty([t_k, t_k +T],\R^{m})$ of the optimization problem
    \begin{equation}\label{eq:RobustFMPCOCP}
            \mathop
            {\operatorname{minimize}}_{\substack{u\in L^{\infty}([t_k, t_k+T],\R^{m}),\\\SNorm{u}\leq M}}  \quad
            \int_{t_k}^{t_k+T}\ell(t,x(t;t_k,\hat{x}_k,u),u(t))\d t.
    \end{equation}
    \item \label{alg:predict}
    Predict the output~$y_{\rm M}$  of the model~\eqref{eq:Mod} on the interval $[t_k,t_{k+1}]$ 
    \[
        y_{\rm M}(t)=h(x(t;t_k,\hat{x}_k,u_{\rm FMPC})),
    \]
    and define the adaptive funnel $\phi: [t_k,t_{k+1}]\to\Rpp $ by 
    \begin{equation} \label{alg:eq:vp}
        \phi(t):=\psi(t)-\Norm{e_{\rm M}(t)}
        \quad\text{ with} \ 
        e_{\rm M}(t) = y_{\rm M}(t) - y_{\rf}(t).
    \end{equation}
    \item\label{alg:step:FC} Define the funnel control law as in~\eqref{eq:u_fc} with $y_{\rm M}$ and funnel function~$\vp$ as in~\eqref{alg:eq:vp} by
    \begin{equation} \label{alg:eq:FC}
        u_{\rm FC}(t) := \beta(\Norm{e_{\rm S}(t)\slash\phi(t)}) (N \circ \alpha)(\Norm{ e_{\rm S}(t)\slash\vp(t)}^2)  e_{\rm S} (t)\slash\vp(t)
        \ \text{ with } \
        e_{\rm S}(t) =y(t)-y_{\rm M}(t).
    \end{equation}
    \item Apply the feedback law
        \begin{equation} \label{eq:u}
            \mu:[ t_k,t_{k+1})\to\R^m, \quad \mu(t) 
            = u_{\rm FMPC}(t)+ u_{\rm FC}(t)
        \end{equation}
        to system~\eqref{eq:Sys}.
        Set the predicted state $x^{\rm pre}_{k+1}=x(t_{k+1};t_k,\hat{x}_k,u_{\rm FMPC})$,
        then increment~$k$ by $1$
        and go to Step~\ref{agostep:RobustFMPCFirst}.
    \end{enumerate}
\end{algo}
Since $\hat x_k$ is chosen via the the proper initialization strategy $\kappa_{\xi}(x^{\rm pre}_{k},\hat y_k)$ at every time instant $t_{k}$ for $k\in\N_0$,
in general $\hat{x}_k\neq x^{\rm pre}_{k}$, see \Cref{Lemma:NonTrivialInit}. In particular, it is possible that $\hat x_0 \neq x^0$.
\begin{remark}
    \Cref{Algo:RobustFMPC} combines the funnel MPC~\Cref{Algo:FMPC} proposed in~\cite{BergDenn21} with the model-free funnel
    control law from~\cite{BergIlch21} by introducing the steps~\ref{alg:predict} and \ref{alg:step:FC}. 
    Using the model output~$y_{\rm M}$ as a reference signal for the funnel controller
    allows the combined controller to benefit from the predictions made by the funnel MPC component, even if the system is in a safety critical state, 
    and ensures that the control signal $u_{\rm FMPC}$ does not act as a disturbance for the funnel controller. 
    In combination with the design of the funnel function~$\phi$, which also depends on the predictions from the MPC component, 
    this combined controller guarantees that the tracking error evolves within the prescribed performance funnel $\psi$, see~\Cref{Thm:RFMPC} below.
    The main mathematical difficulty lies in ensuring that the funnel MPC algorithm remains feasible, if the model is initialized with measurements of the system output via the initialization strategy~$\kappa_\xi$,
    and in adapting the results from~\cite{BergIlch21} to the current setting.
    The findings in~\cite{BergIlch21} cannot be directly applied, since the reference signal for the funnel controller is assumed to be a-priori given and to be continuous;
    both assumptions are not met in~\Cref{Algo:RobustFMPC}.
\end{remark}

\subsection{Main result}\label{Sec:MainResult}

In the following main result we show that the robust funnel MPC \Cref{Algo:RobustFMPC} is initially and recursively feasible and achieves  tracking of a given reference signal with prescribed behavior.
\begin{theorem} \label{Thm:RFMPC}
    Consider a system~\eqref{eq:Sys} with $(d,F,\oT) \in \cN^m$ as in~\Cref{Def:system-class} and choose a model~\eqref{eq:Mod} with $(f,g,h) \in \cM^m$ as in~\Cref{Def:Model_Class}.
    Let $\psi \in \cG$ and $y_{\rm ref} \in W^{1,\infty}(\Rp,\R^m)$ be given.
    Let  $y^0\in\con([-\sigma,0],\R^m)$ with $\sigma\geq 0$ be an initial {history function}  for system~\eqref{eq:Sys} with $y^0(0)\in\cD_0$.
    Then, for any $\xi\ge 0$, the set~$X^0$ in~\eqref{eq:ModelInitalValues} is non-empty and
    there exists $M>0$ such that the robust funnel MPC \Cref{Algo:RobustFMPC} with $\delta>0$ and $T\ge\delta$ is initially and recursively feasible for every $x^0 \in X^0$, 
    i.e., 
    \begin{itemize}
        \item at every time instance $t_k := k\delta $ for $k\in\N_0$ the OCP~\eqref{eq:RobustFMPCOCP} has a solution $u_k^*\in L^\infty([t_k,t_k+T],\R^m)$, and
        \item the closed-loop system consisting of the system~\eqref{eq:Sys} and the feedback law~\eqref{eq:u} 
    has a global solution $y : [-\sigma,\infty) \to \R^m $. 
    \end{itemize}
    Each global solution $y$ satisfies that
    \begin{enumerate}[label = (\roman{enumi}), ref=(\roman{enumi})]
        \item \label{Assertion:y_u_bounded}
        all signals are bounded, in particular, $u \in L^\infty(\Rp, \R^m)$ and $y \in L^\infty([-\sigma,\infty), \R^m)$, 
        \item \label{Assertion:tracking_error}
        the tracking error between the system's output and the reference evolves within prescribed boundaries, i.e.,
        \begin{equation*}
            \forall \, t \ge 0 \, : \ \|y(t)  - y_{\rm ref}(t) \| < \psi(t) .
        \end{equation*}
    \end{enumerate}
\end{theorem}
The proof is relegated to \Cref{Sec:proofs}.

\begin{remark}
   With \Cref{Thm:RFMPC} at hand we comment on the difference between the proposed control scheme and a straightforward combination of a MPC scheme with a feedback control law.
    \begin{enumerate}[label = (\roman*), leftmargin=*]
        \item The combination of feedforward control with feedback control, i.e., the two degree of freedom controller design~\cite{SkogPost05}, is a popular approach. The specific combination of funnel control with feedforward control methods was investigated in~\cite{BergOtto19,BergDrue21}. 
        In a similar fashion it is possible to combine a MPC scheme, in particular, a funnel MPC scheme, with an additional feedback controller.
        This possibility (i.e., no feedback between funnel MPC and the actual system) is realized in the robust funnel MPC~\Cref{Algo:RobustFMPC} by allowing that the initialization at the beginning of a MPC cycle consists only of the previous prediction~$x_k^{\rm pre}$ of the current 
        \textit{model} state such that $\hat x_k := x_k^{\rm pre}$, which is a special instance of \emph{proper initialization}.
        In this case, the funnel MPC control signal $u_{\rm FMPC}$ can be computed offline using the given model.
        Then it is applied to the system as an open-loop control and the additional feedback control compensates errors,
        which occur due to deviations between the model and the system.
        This situation is illustrated as the second scenario in the simulation in \Cref{sec:sim}, cf.~\Cref{Fig:Chem_reac_noreinit_robust}.
        \item The alternative to the open-loop operation of \Cref{Algo:RobustFMPC} is a feedback induced by the utilization of measurements of the \textit{system} output $\hat y_k := y(t_k)$.
        Then, initializing the model properly with $\hat x_k \in \Omega_{\xi}(x_k^{\rm pre}, \hat y_k)$ ensures recursive feasibility of the MPC scheme on the one hand, 
        and on the other hand, the state $\hat x_k$ is chosen such that the model's output $h(\hat x_k)$ equals the 
        system's output $\hat y_k$, i.e., $ h(\hat x_k) = \hat y_k$.
        With this re-initialization at the beginning of the MPC cycle, the influence of the control signal $u_{\rm FMPC}$
        to the system is taken into account, and moreover, since the error between the model's and the system's output is zero,
        the optimal control signal may have a better effect on the system's tracking behavior.
        This situation is illustrated in the third scenario in the simulation in \Cref{sec:sim}, cf.~\Cref{Fig:Chem_reac_reinit_robust}.
    \end{enumerate}
\end{remark}

\begin{remark} \label{cor:u_bound}
 If the model~\eqref{eq:BIF} and the system~\eqref{eq:Sys} coincide up to an additive bounded disturbance, then we may derive an explicit bound for the overall control input $u = u_{\rm FMPC} + u_{\rm FC}$ a-priori.
    Consider a model~\eqref{eq:BIF} and let the system be given by
    \begin{align*}
         \dot y(t) &= p(y(t),\eta_{\rm S}(t)) + \Gamma(\Phi^{-1}(y(t),\eta_{\rm S}(t))) \, u(t) + d(t), && y(0) = y^0, \\
        \dot \eta_{\rm S}(t) &= q(y(t),\eta_{\rm S}(t)), && {\eta_{\rm S}(0) = \eta_{\rm S}^0,}
    \end{align*}
    where $d \in L^\infty(\Rp,\R^m)$ and the high-gain matrix satisfies $\Gamma(x)+ \Gamma(x)^\top > 0$ for all $x\in\R^n$.
    In this case the surjection in the funnel controller can be {replaced with the function} $N(s) = -s$, whereby the funnel control law simplifies to 
    $u_{\rm FC}(\cdot) = - \beta(\|w(\cdot)\|) \alpha(\|w(\cdot)\|^2) w(\cdot)$, $w := (y-y_{\rm M})/\vp$, cf. \cite[Rem.~1.8]{BergIlch21}.
    Let $\ve \in (0,1)$ be the smallest number such that
    \begin{align*}
        \beta(\ve) \alpha(\ve^2) \ve^2 = 
        \frac{ \|\dot \psi\|_\infty + 3 \max_{(y,\eta) \in K} \|p(y,\eta)\| +
        3 \max_{(y,\eta) \in K}\| \Gamma(\Phi^{-1}(y,\eta))\| M +
        \|d\|_\infty}{ \lambda_\Gamma },
    \end{align*}
    where $\lambda_\Gamma > 0$ is the smallest eigenvalue of $\Gamma(\cdot) + \Gamma(\cdot)^\top$ on $\Phi^{-1}(K)$, and the compact set $K$ is given in the proof of \Cref{Prop:OCP_has_solution}.
    Then, invoking the same arguments as in \emph{Step three} in the proof of \Cref{Thm:RFMPC}, the overall control satisfies
    \[
     \| u \|_\infty \le M + \beta(\ve) \alpha(\ve^2) .
    \]
    Note that although the bound on the control $u$ is explicitly given, this bound involves some non-trivial computations such as deriving the compact set $K$ in \Cref{Prop:OCP_has_solution} explicitly and computing the maximal values of the system parameters on this compact set. Moreover, the bound is conservative in the sense that in applications the maximal input will typically be much smaller.
    \end{remark}

\section{Simulation}\label{sec:sim}
We illustrate the application of the proposed control strategy \Cref{Algo:RobustFMPC} by a numerical simulation.
To this end, we consider a continuous chemical reactor and concentrate on the control goal to steer the reactor's 
temperature to a certain given value $y_{\rm ref}(t)$ within boundaries given by a function~$\psi(t)$.
The reactor's temperature should follow a given heating profile specified as
\begin{equation*}
    y_{\rm ref}(t) = 
    \begin{cases}
        y_{\rm ref,start} + \frac{y_{\rm ref,final} - y_{\rm ref,start}}{t_{\rm final}} t, & t \in [0,t_{\rm final}), \\
        y_{\rm ref,final}, & t \ge t_{\rm final} . 
    \end{cases}
\end{equation*}
Note that this heating profile has a kink at~$t = t_{\rm final}$.
Starting at $y_{\rm ref,start} = 270 \, \rm K$, the reactor is heated up to $y_{\rm ref,final} = 337.1 \, \rm K$ within the prescribed time~$[0,t_{\rm final}]$, here we choose~$t_{\rm final} = 2$.
During the heating phase, the tolerated temperature deviation from the heating profile decreases from~$\pm 24 \, \rm K$ to~$\pm 4.4 \, \rm K$ (time-varying output constraints).
After reaching the desired level, the temperature in the reactor is kept constant with deviation of no more than $\pm 4.4 \, \rm K$ after four minutes after beginning of the heating process.
In the reactor the first order and exothermic reaction Substance-1 $\to$ Substance-2 takes place.
Such a reactor can be modeled by the following system of equations, cf.~\cite{VielJado97a}
\begin{equation} \label{eq:chem_react_system}
    \begin{aligned}
        \dot y(t) & = b p(x_1(t),x_2(t),y(t)) - q y(t) + u(t), \\
        \dot x_1(t) &= c_1 p(x_1(t),x_2(t),y(t)) + d (x_1^{\rm in} - x_1(t)), \\
        \dot x_2(t) &= c_2 p(x_1(t),x_2(t),y(t)) + d (x_2^{\rm in} - x_2(t)), 
    \end{aligned}
\end{equation}
where $x_1$ is the concentration of the reactant Substance-1, $x_2$ the concentration of the product Substance-2 and $y$ describes the reactor temperature; $u$ is the feed temperature/coolant control input.
The value $x_i^{\rm in}$ is the (positive) concentration of Substance-$i$ ($i=1,2$) in the feed flow.
Further, the constant $b > 0$ describes the exothermicity of the reaction, $d>0$ is associated with the dilution rate and $q > 0$ is a constant consisting of the combination of the dilution rate and the heat transfer rate.
Further, $c_1,c_2 \in \R$ are the stoichiometric coefficients and $p : \Rp \times \Rp \times \Rp \to \Rp$ is the reaction heat; here the latter involves the Arrhenius function and is assumed to be given as
\begin{equation*} 
    p(x_1,x_2,y) = k_0 e^{- \frac{k_1}{y}} x_1,
\end{equation*}
where $k_0, k_1$ are positive parameters. As a model for this nonlinear reaction process we consider a linearization of system~\eqref{eq:chem_react_system}, obtained by linearizing the Arrhenius function around the desired final temperature $\bar y = 337.1K$ and $x_1 = \tfrac12 x_1^{\rm in}$.
This results in
\begin{equation*}
    p_{\rm lin}(x_1,x_2,y) = k_0 e^{-\frac{k_1}{\bar y}} x_1 + \frac{k_0 k_1 e^{-\frac{k_1}{\bar y}}}{\bar y^2}  \frac{x_1^{\rm in}}{2}  (y-\bar y).
\end{equation*}
We set $a_1 := \tfrac{k_0 k_1 e^{-\frac{k_1}{\bar y}}}{\bar y^2} \tfrac{x_1^{\rm in}}{2}$, $a_2 := k_0 e^{-\frac{k_1}{\bar y}}$ and define the expressions
\begin{equation*}
    A = \begin{bmatrix}
    b a_1 - q & b a_2 & 0 \\
    c_1 a_1 & c_1 a_2 - d & 0 \\
    c_2 a_1 & c_2 a_2 & -d
    \end{bmatrix}\in\R^{3\times3}, \quad
    D = \begin{bmatrix}
        - b a_1 \bar y \\
        -c_1 a_1 \bar y + d x_{1,\rm M}^{\rm in} \\
        -c_2 a_1 \bar y + d x_{2,\rm M}^{\rm in}
    \end{bmatrix} \in \R^{3}.
\end{equation*}
Then, with $x := (y_{\rm M},x_{1,\rm M}, x_{2,\rm M})^\top \in \R^3$ the model is given by
\begin{equation*}
\begin{aligned}
    \dot x(t) &= A x(t) + B u_{\rm FMPC}(t) + D, \\
    y_{\rm M}(t) &= C x(t),
\end{aligned}
\end{equation*}
where  $C = B^\top = [1,0,0] \in \R^{1\times 3}$.
We run the simulation on an interval of $[0, 4]$ minutes, and choose according to \cite{VielJado97a,IlchTren04} the following values for the parameters:
$c_1 = -1 = -c_2$, $k_0 = e^{25}$, $k_1 = 8700$, $d=1.1$, $q=1.25$, $x_1^{\rm in} = 1$, $x_2^{\rm in}=0$ and $b = 209.2$  and initial values $[y(0), x_1(0),x_2(0)]=[y_{\rm M}(0), x_{{\rm M},1}(0), x_{{\rm M},2}(0)]= [270,0.02,0.9]$. The funnel function is given by $\psi(t):=20\me^{-2t}+4$.
We simulate the following scenarios:
\begin{itemize}
    \item \textit{Case 1:} Funnel MPC without robustification, i.e., $u_{\rm FMCP}$ is computed via \Cref{Algo:FMPC} and applied to the system without an additional funnel control loop; this is shown in \Cref{Fig:Chem_reac_noreinit_norobust}.
    \item \textit{Case 2:} Robust funnel MPC with trivial proper re-initialization, i.e., $\hat x_k = x_k^{\rm pre}$ in Step~(a) of \Cref{Algo:RobustFMPC}; this is depicted in~\Cref{Fig:Chem_reac_noreinit_robust}.
    \item \textit{Case 3:} Robust funnel MPC with proper initialization according to the system's output, i.e., $\hat x_k \in \Omega_{\xi}(x_k^{\rm pre}, \hat y_k)$ such that $h(\hat x_k) = \hat y_k$ in Step~(a) of \Cref{Algo:RobustFMPC}; this is shown in~\Cref{Fig:Chem_reac_reinit_robust}.
\end{itemize}
As activation function we take the ReLU-like map
\begin{equation*}
\beta(s) = 
    \begin{cases}
        0, & s\le S_{\rm crit}, \\
        s- S_{\rm crit}, & s \ge S_{\rm crit},
    \end{cases}
\end{equation*}
where we choose $S_{\rm crit} = 0.5$, i.e., the funnel controller becomes active, if the error $y-y_{\rm M}$ exceeds $50 \%$ of the maximal distance to its funnel boundary.
In this example, we restrict the MPC control signal to $\| u_{\rm FMPC} \|_\infty \le 600$.
The input constraint is indicated by a dotted line in \Cref{Fig:Chem_reac_noreinit_norobust,Fig:Chem_reac_noreinit_robust,Fig:Chem_reac_reinit_robust}.
Further, we choose the design parameters $\lambda_u = 10^{-4}$, prediction horizon $T = 0.75$, and time shift $\delta = 0.05$.
In the following figures, the control signal generated via funnel MPC is labeled with the subscript~FMCP ($u_{\rm FMPC}$); the signal generated by the additional funnel controller is labeled with the subscript~FC ($u_{\rm FC}$).
\begin{figure}
\centering
    \begin{minipage}[t]{0.48\textwidth}
    \centering
   \includegraphics[width=\textwidth]{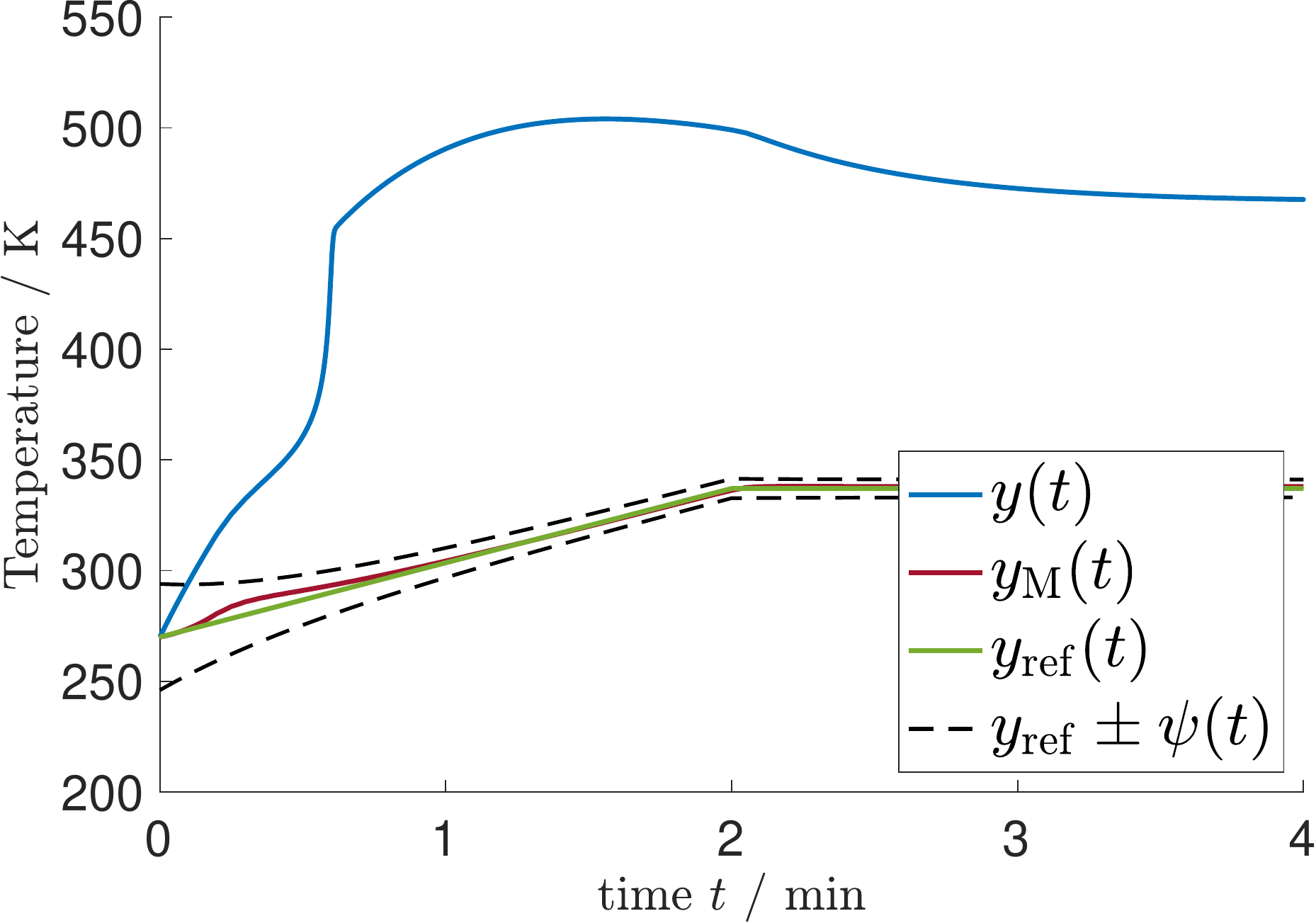}
\end{minipage}
\begin{minipage}[t]{0.48\textwidth}
\centering
    \includegraphics[width=\textwidth]{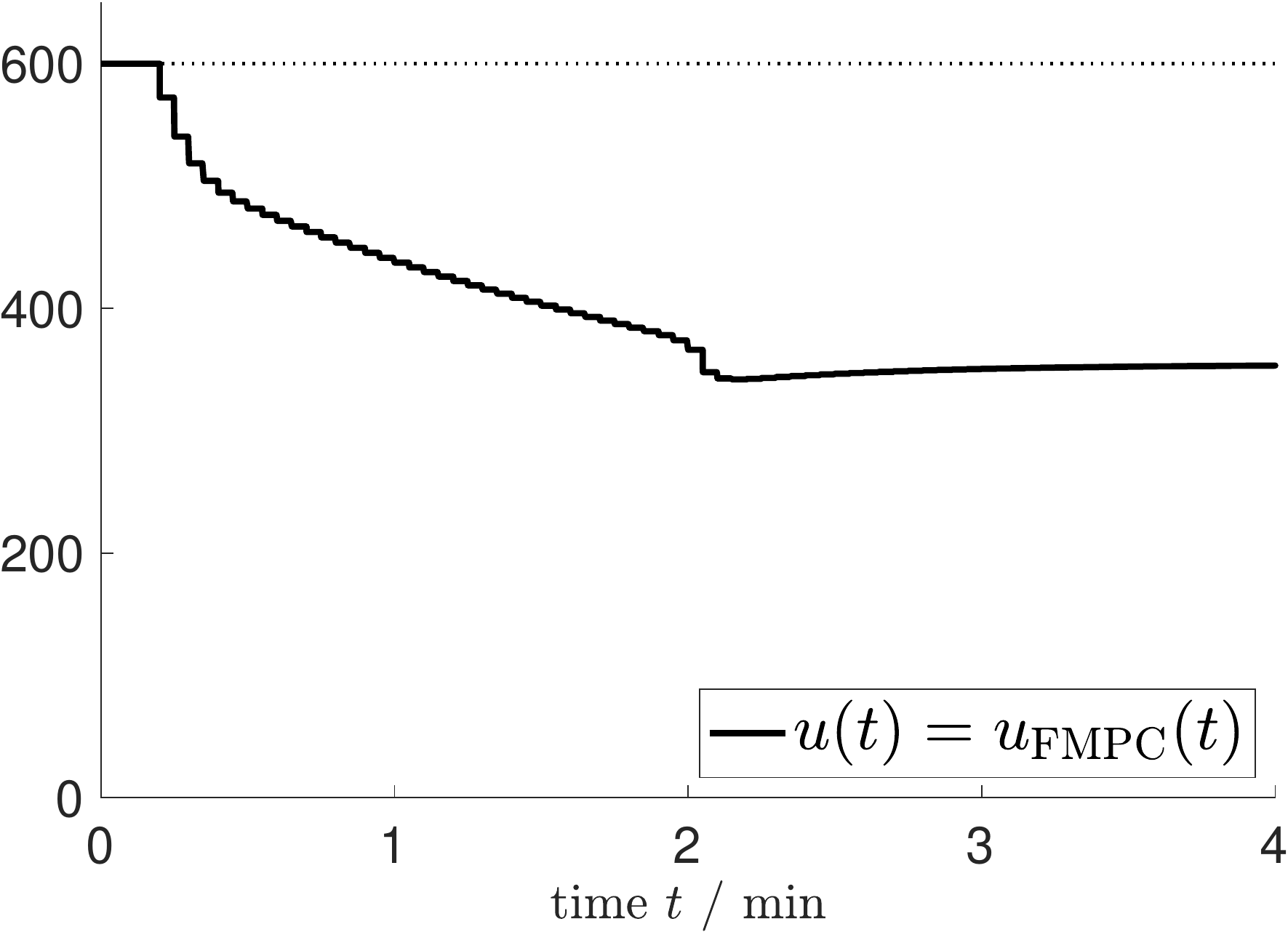}
\end{minipage}
    \caption{Application of the control computed by robust funnel MPC without additional funnel control feedback loop.}
\label{Fig:Chem_reac_noreinit_norobust}
\end{figure}
\Cref{Fig:Chem_reac_noreinit_norobust} shows the application of the control signal computed with funnel MPC \Cref{Algo:FMPC} in Case~1 to the system without an additional funnel control feedback loop. The error $e_{\rm M}(t) = y_{\rm M}(t) - y_{\rm ref}(t)$ between the model's output $y_{\rm M}(t)$ and the reference $ y_{\rm ref}(t)$ evolves within the funnel boundaries $\psi(t)$.
However, the control signal computed with funnel MPC using the linear model is not sufficient to achieve that the tracking error $e(t) = y(t) - y_{\rm ref}(t)$ evolves within the funnel boundaries $\psi(t)$.
Obviously, the deviation is induced during the initial phase.
After about two minutes, the linearized model is a good approximation of the system. In this region, the control $u_{\rm FMPC}$ has a similar effect on both dynamics; however, the error $y(t)- y_{\rm ref}(t)$ already evolves outside the funnel boundaries $\psi(t)$.
\begin{figure}
\centering
    \begin{minipage}[t]{0.48\textwidth}
    \centering
   \includegraphics[width=\textwidth]{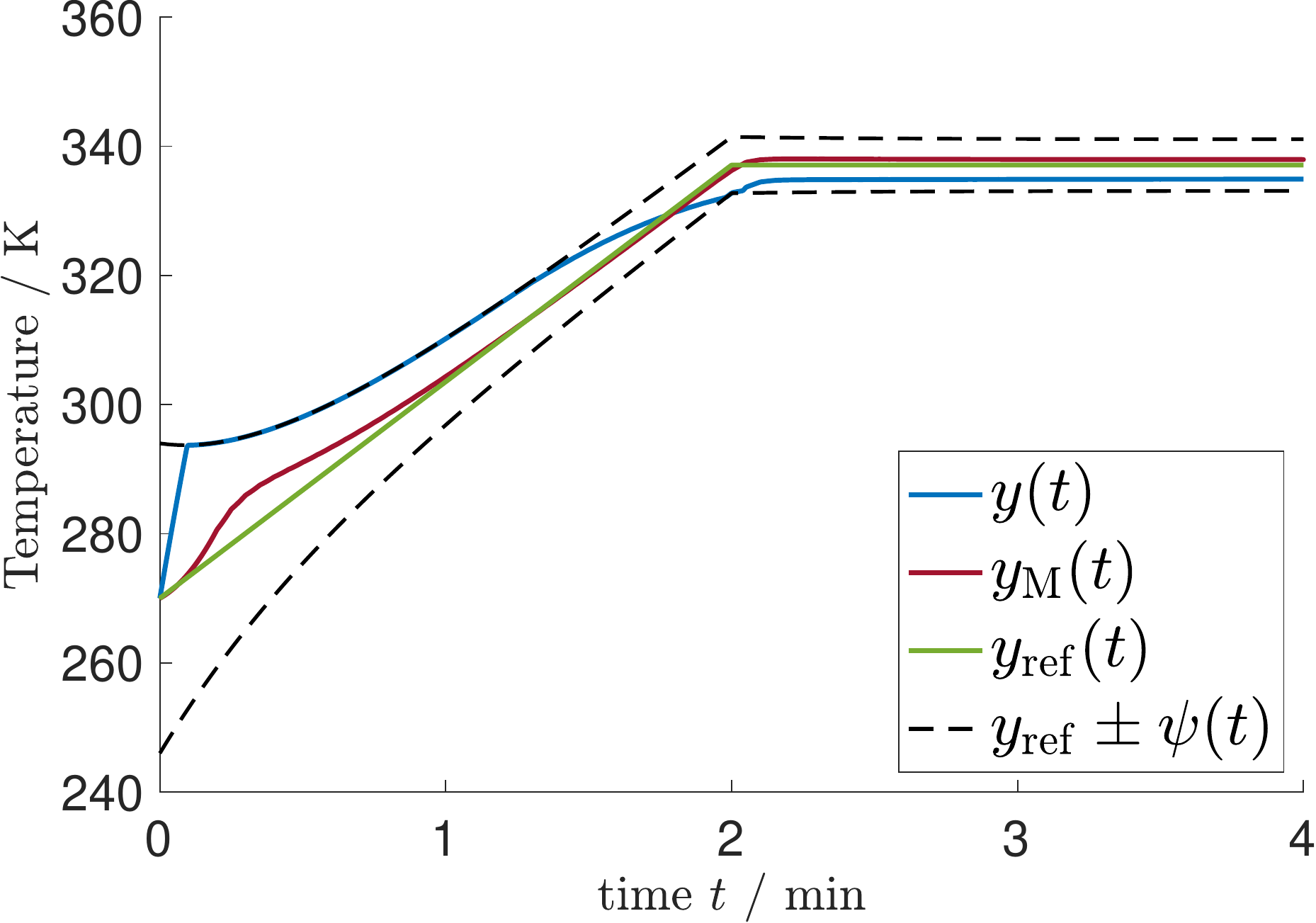}
\end{minipage}
\begin{minipage}[t]{0.48\textwidth}
\centering
    \includegraphics[width=\textwidth]{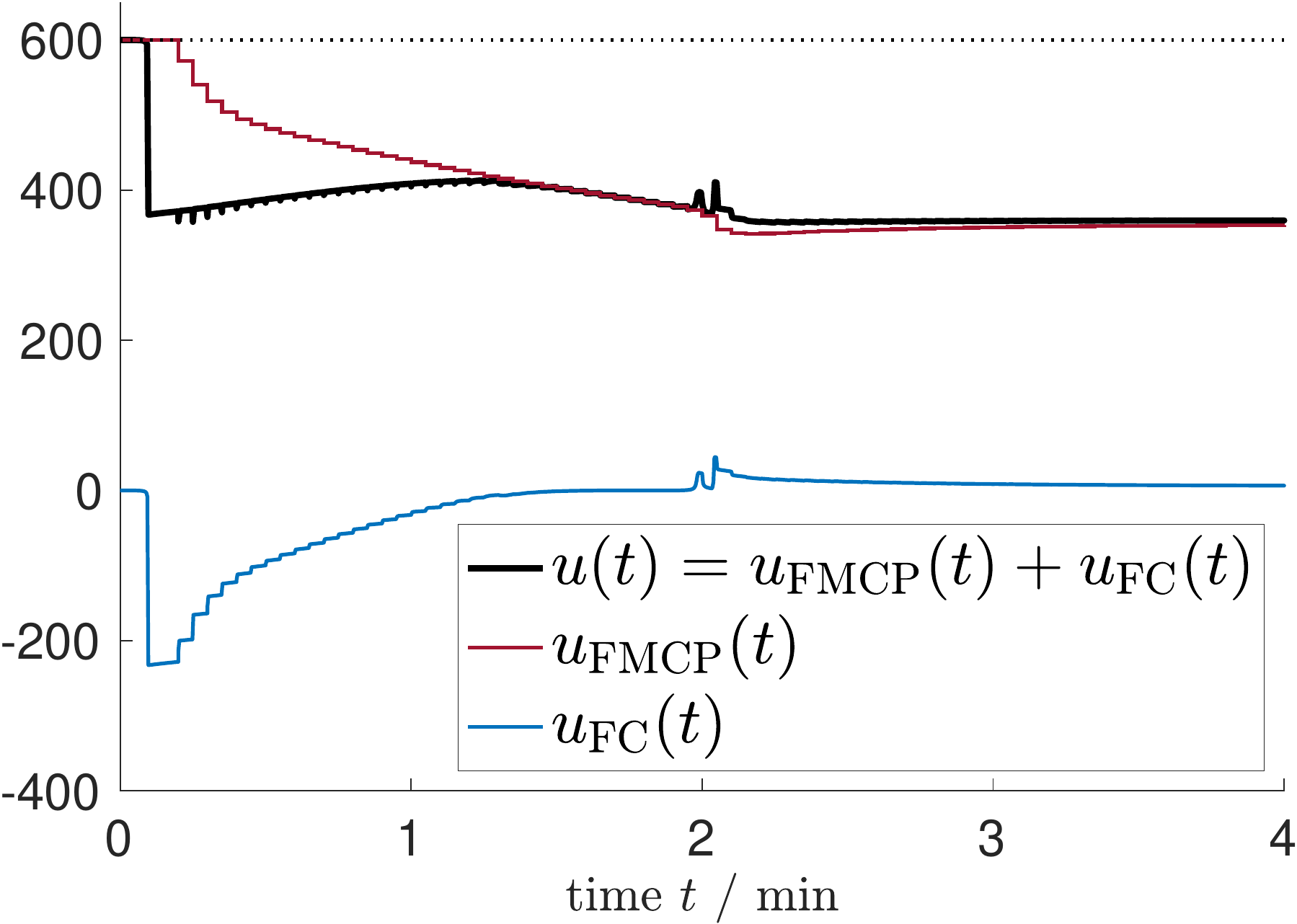}
\end{minipage}
    \caption{Application of the control computed by robust funnel MPC with additional funnel control feedback loop, without re-initialization.}
    \label{Fig:Chem_reac_noreinit_robust}
\end{figure}
\Cref{Fig:Chem_reac_noreinit_robust} shows the application of the control signal computed with robust funnel MPC \Cref{Algo:RobustFMPC} in Case~2, i.e., besides the funnel MPC control signal the additional funnel controller is applied in order to guarantee that the error $y(t) - y_{\rm ref}(t)$ evolves within the boundaries $\psi(t)$.
Since the model and the system do not coincide, the system evolves differently and hence the funnel controller has to compensate the model-plant mismatch.
\begin{figure}
\centering
    \begin{minipage}[t]{0.48\textwidth}
    \centering
   \includegraphics[width=\textwidth]{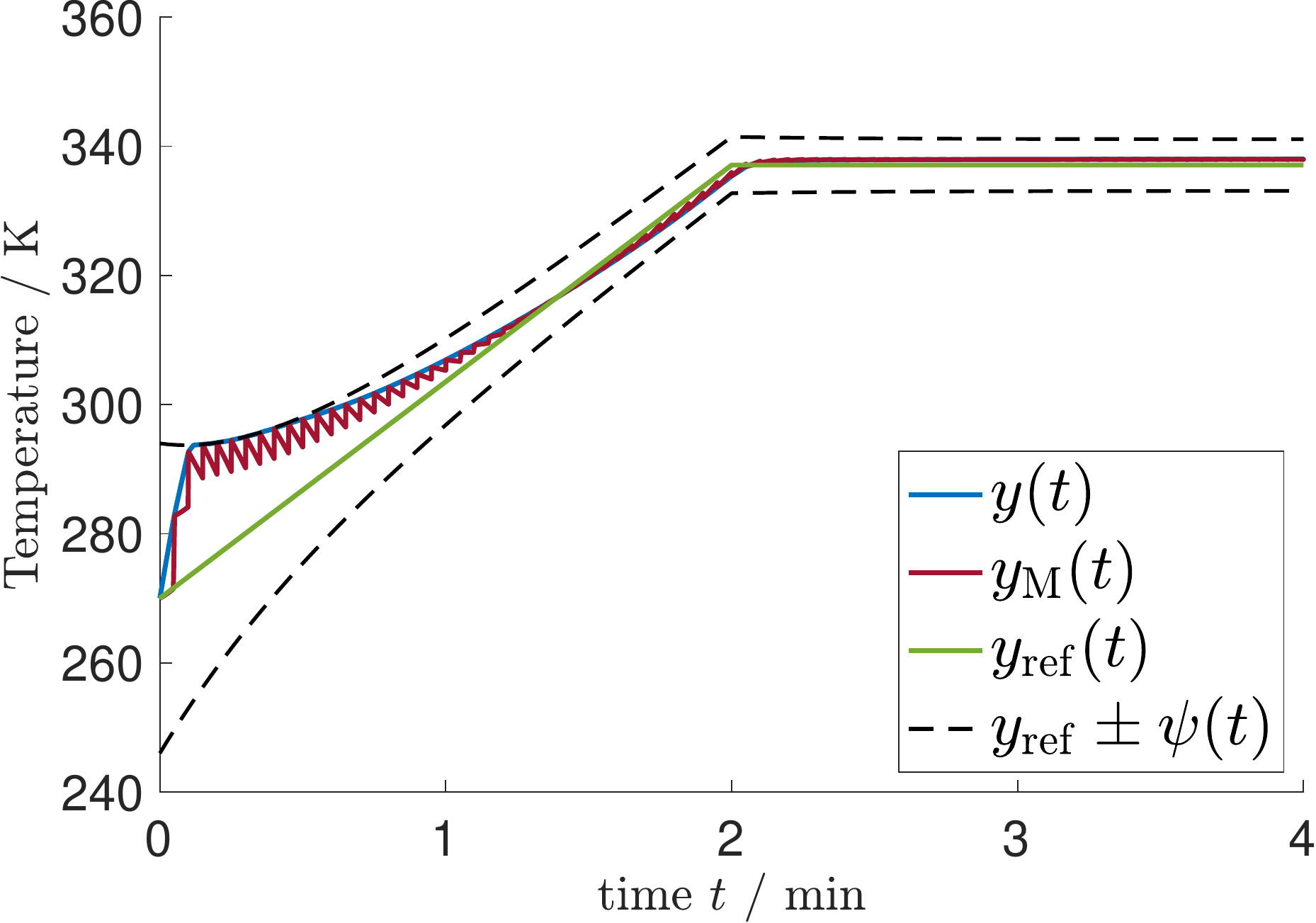}
\end{minipage}
\begin{minipage}[t]{0.48\textwidth}
\centering
    \includegraphics[width=\textwidth]{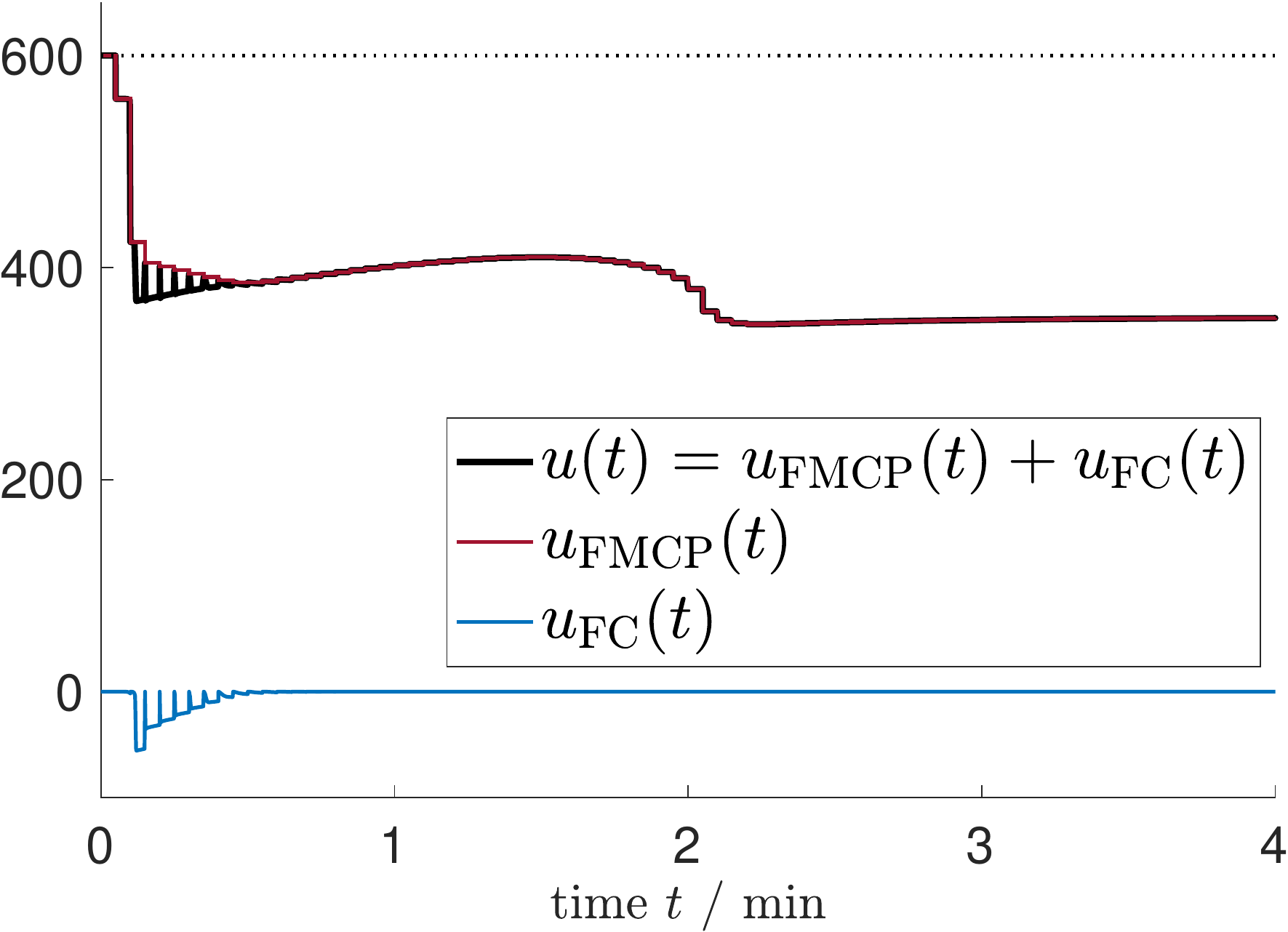}
\end{minipage}
    \caption{Application of the control computed by robust funnel MPC with additional funnel control feedback loop and with re-initialization.}
    \label{Fig:Chem_reac_reinit_robust}
\end{figure}
\Cref{Fig:Chem_reac_reinit_robust} shows the application of \Cref{Algo:RobustFMPC} in Case~3. Besides the additional application of the funnel controller, the model's state is updated with $h(\hat x_k) = \hat y_k$ at the beginning of every MPC cycle.
Note that 
in~\Cref{Fig:Chem_reac_reinit_robust} the funnel controller is inactive most of the time, i.e., the applied control signal can be viewed to be close to \emph{optimal} with respect to the cost function~\eqref{eq:stageCostFunnelMPC}, since it is computed via the OCP~\eqref{eq:RobustFMPCOCP}.

\section{Conclusion and Outlook} \label{Sec:Conclusion}

We proposed a two-component controller to achieve output reference tracking with prescribed
accuracy for unknown nonlinear systems.
One component is an MPC scheme, which uses a particular stage cost to guarantee reference tracking within prescribed error margins of the underlying model.
The MPC controller is safeguarded by the adaptive funnel controller, 
which compensates the error when necessary.
We rigorously proved initial and recursive feasibility of the combined 
\Cref{Algo:RobustFMPC}.
A key
feature of the proposed two-component controller is the possibility to update the model's state based on
measured system data during runtime.
This feature 
is realized via the newly introduced \emph{proper initialization strategy}~$\kappa$.
It allows to initialize the model in the optimal control problem based on
system data such that the resulting control is more likely to affect the real plant in the intended way.

Future research will focus on extracting criteria to find explicit and beneficial proper initialization strategies to take advantage of available measurement data. 
Moreover, we aim to exploit system data to improve the underlying model. 
In the recent work~\cite{lanza2023safe} first steps in this direction were initiated.

\section{Appendix} \label{Sec:proofs} 
Before we present the main proof, we establish some auxiliary results.
In \Cref{Sec:ExistenceSolOPC} we show the existence of a solution of the OPC~\eqref{eq:RobustFMPCOCP}.
In \Cref{Sec:FCResults} we state some results concerning the application and combination of the funnel controller with the MPC scheme.
Finally in \Cref{Sec:ProofMain} we provide a proof of the main result \Cref{Thm:RFMPC}.

\subsection{Existence of an optimal control} \label{Sec:ExistenceSolOPC}
The first proposition concerns the existence of a solution of the OCP~\eqref{eq:RobustFMPCOCP}.
\begin{prop} \label{Prop:OCP_has_solution}
    Consider the model~\eqref{eq:Mod} with $(f,g,h)\in\cM^m$ as in~\Cref{Def:Model_Class}.
    Let $\delta>0$, $\xi\geq 0$, $\psi\in\cG$, ~$y_{\rf}\in W^{1,\infty}(\Rp,\R^m)$
    and a proper initialization strategy~$\kappa_{\xi}:\R^n\times\R^m\to\R^n$  be given.
    Let the sequence $(t_k)_{k\in\N_0}$ be defined by $t_k=k \delta$ and 
    $(\hat{y}_k)_{k\in\N_0}$ be an arbitrary sequence with $\hat y_k\in\cD_{t_k}$ for all~$k\in\N_0$.
    Then there exists~$M>0$, independent of~$\delta$, such that for all $T\geq\delta$ and all
    $x_{0}^{\rm pre}\in
    \setdef{x\in\R^n}
    {
        h(x)\in\cD_{t_0},\ \Norm{[0,I_{n-m}]\Phi(x)} \le \xi
    }$ 
    the~OCP
    \begin{equation}\label{eq:Prop:OPC}
            \mathop
            {\operatorname{minimize}}_{\substack{u\in L^{\infty}([t_k, t_k+T],\R^{m}),\\\SNorm{u}\leq M}}  \quad
            \int_{t_k}^{t_k+T}\ell(t,x(t;t_k,\kappa_{\xi}(x_k^{\rm pre},\hat{y}_{k}),u),u(t))\d t
    \end{equation} 
    has a solution $u_k^\star\in L^{\infty}([t_k, t_k+T],\R^{m})$ for all~$k\in\N_0$,
    where $(x_{k}^{\rm pre})_{k\in\N_0}$ is defined by
    \[
        x_{k+1}^{\rm pre}:=x(t_{k+1};t_k,\kappa_{\xi}(x_k^{\rm pre},\hat{y}_{k}),u^\star_{k}).
    \]
    Moreover, the piecewise continuous function
    \begin{equation*}
        y_{\rm M} : \Rp \to \R^m, \quad t \mapsto \sum_{k \in \N_0} h(x(t;t_k,\kappa_{\xi}(x_k^{\rm pre},\hat{y}_{k}),u^\star_{k}))|_{[t_k,t_{k+1})}
    \end{equation*}   
     satisfies $y_M(t)\in\cD_t$ for all $t\ge 0$ and there exists~$\bar{\lambda}>0$ such that $\esssup_{t \ge 0}\Norm{\dot{y}_M(t)}\leq \bar{\lambda}$.
     The bound~$\bar{\lambda}$ is independent of~$(\hat{y}_k)_{k\in\N_0}$, $\delta$, $x_0^{\rm pre}$ and~$\kappa_{\xi}$.
\end{prop}

\begin{proof}
\emph{Step one.} We introduce some notation.
We denote by $\cY_{\hat{\zeta}}(I)$ the set of all functions $\zeta\in\cR(I,\R^m)$ which start at $\hat{\zeta}\in\R^m$ and 
$\zeta-y_{\rf}$ evolves  
within the funnel given by~$\psi$ on an interval $I\subseteq\R_{\ge 0}$ of the form $I=[a,b]$ 
with $b\in(a,\infty)$ or $I=[a,b)$ with $b=\infty$:
\[
    \cY_{\hat{\zeta}}(I) := \setdef
            {\zeta\in \cR(I,\R^m)}
            {\zeta(\inf I) = \hat{\zeta},\ \fa t\in I:\ \zeta(t)\in\cD_t}.
\]                                      
Recall that 
for $k\in\N_0$, $\hat{\eta}\in\R^{n-m}$ and $\zeta\in \cR([t_k,\infty),\R^m)$,
 $\eta (\cdot;t_k, \hat{\eta},\zeta):[t_k,\infty)\to\R^{n-m}$ denotes the global solution of the initial value problem~\eqref{eq:zero_dyn}, $\eta(t_k)=\hat{\eta}$, where~$y_{\rm M}$ is substituted by~$\zeta$.
Define for $k\in\N_0$ and $t\geq t_k$ the set
\[
    N^{t}_{t_k}:=
    \setdef
    {\eta(t;t_k,\hat{\eta},\zeta)}
    {(\hat{\zeta},\hat{\eta})\in\cD_{t_k}\times{\dB}_{\xi},\ \zeta\in \cY_{\hat{\zeta}}([t_k,\infty)) },
\]
where $\dB_{\xi}:=\setdef{z\in\R^{n-m}}{\Norm{z}\leq\xi}$. Finally, define  the set
\[
    \cU(t_k,\hat{x}):= \setdef
    {u\in L^\infty([t_k,t_k+T],\R^m)}
    {
       \! \begin{array}{l}
            \SNorm{u}\leq M,\\
            \fa t\in [t_k,t_k+T]:\ h(x(t;t_k,\hat{x},u))\in\cD_t
        \end{array}\!\!
    }
\]
of all~$L^\infty$--controls~$u$ bounded by~$M>0$ which, if applied to the model~\eqref{eq:Mod},
guarantee that the error $e_{\rm M}=y_{\rm M}-y_{\rf}$ evolves within the funnel~$\cF_{\psi}$ on the interval~$[t_k,t_{k+1}]$.

\noindent
\emph{Step two.} For arbitrary~$k\in\N_0$, we make three observations:
\begin{enumerate}[label = (\roman{enumi})]
\item\label{Item:PredictionInitVal}
    Since~$\kappa_{\xi}$ is a proper initialization strategy and $\hat{y}_k\in\cD_{t_k}$ for all~$k\in\N_0$, the following holds: 
    \[
    x_k^{\rm pre}\in \Phi^{-1}
            \rbl
            \cD_{t_k}\times
            \bigcup_{i=0}^{k}N_{t_i}^{t_k}
            \rbr 
            \quad \Impl\quad
            \kappa_{\xi}(x_k^{\rm pre},\hat{y}_{k})\in\Phi^{-1}
            \rbl 
            \cD_{t_k}\times
            \bigcup_{i=0}^{k}N_{t_i}^{t_k}
            \rbr.
    \]
\item\label{Item:RecInitialVal}
    If, for~$\hat{x} \in 
    \Phi^{-1}
            \rbl
            \cD_{t_k}\times
            \bigcup_{i=0}^{k}N_{t_i}^{t_k}
            \rbr$,
    the set~$\cU(t_k,\hat{x})$ is non-empty and an element~$u\in\cU(t_k,\hat{x})$ is applied to the model~\eqref{eq:Mod}, then 
    \[
       x(t_{k+1};t_k,\hat{x},u)\in
      \Phi^{-1}
            \rbl
            \cD_{t_{k+1}}\times
            \bigcup_{i=0}^{k+1}N_{t_i}^{t_{k+1}}
            \rbr.
    \]
\item\label{Item:SolutionOptimizationProb}
    If, for~$\hat{x}\in\R^n$, the set $\cU(t_k,\hat{x})$ is non-empty, then the OCP
    \[
        \mathop
        {\operatorname{minimize}}_{\substack{u\in L^{\infty}([t_k, t_k+T],\R^{m}),\\\SNorm{u}\leq M}}  \quad
        \int_{t_k}^{t_k+T}\ell(t,x(t;t_k,\hat{x}),u(t))\d t
    \]
    has a solution~$u_k^{\star}\in\cU(t_k,\hat{x})$  according to~\cite[Thm.~4.6]{BergDenn21}.
\end{enumerate}
For~\ref{Item:PredictionInitVal} observe that $[I_m,0]\Phi(x) = h(x)$ for all $x\in\R^n$. 
To see \ref{Item:RecInitialVal}, let $\hat{x} \in 
    \Phi^{-1}
            \rbl
            \cD_{t_k}\times
            \bigcup_{i=0}^{k}N_{t_i}^{t_k}
            \rbr$
    be such that  $\cU(t_k,\hat{x})$ is non-empty. If $u\in\cU(t_k,\hat{x})$ is applied to the model~\eqref{eq:Mod}, then 
    $h(x(t;t_k,\hat{x},u))\in\cD_t$ for all $t\in [t_k,t_k+T]$, in particular $h(x(t_{k+1};t_k,\hat{x},u))\in\cD_{t_{k+1}}$.
    Furthermore, there exists~$i\leq k$ such that $[0, I_{n-m}] \Phi(\hat x)\in N_{t_i}^{t_k}$ and hence there exist $(\hat{\zeta},\hat{\eta})\in\cD_{t_i}\times{\dB}_{\xi}$ and $\zeta\in\cY_{\hat{\zeta}}([t_i,\infty))$ with
        $[0,I_{n-m}]\Phi(\hat{x})= \eta(t_{k};t_i,\hat{\eta},\zeta)$. Define $\tilde{\zeta}:[t_i,\infty)\to \R^{m}$ by 
    \[
        \tilde{\zeta}(t) :=
        \begin{cases}
           h(x(t;t_k,\hat{x},u)),&t\in[t_k,t_{k+1}]\\ 
           \zeta(t),&t\in[t_i,t_k)\cup(t_{k+1},\infty).
        \end{cases}
    \]
    Then, $\tilde{\zeta}\in\cY_{\hat{\zeta}}([t_i,\infty))$ and $\eta(t_{k+1};t_i,\hat{\eta},\tilde{\zeta})\in N_{t_i}^{t_{k+1}}$.
    Thus,
    \begin{align*}
        \Phi(x(t_{k+1};t_k,\hat{x},u)) =
        \begin{pmatrix} h(x(t_{k+1};t_k,\hat{x},u))\\ [0,I_{n-m}]\Phi(x(t_{k+1};t_k,\hat{x},u))\end{pmatrix} =\begin{pmatrix}
        h(x(t_{k+1};t_k,\hat{x},u)) \\ \eta(t_{k+1};t_i,\hat{\eta},\tilde{\zeta})\end{pmatrix}
        \in
            \cD_{t_{k+1}}\times
            N_{t_i}^{t_{k+1}}.
    \end{align*}
\noindent
\emph{Step three.} 
Since $\psi\in W^{1,\infty}(\Rp,\R)$ and $y_{\rf}\in W^{1,\infty}(\Rp,\R^m)$ are bounded,
the set $O:=\bigcup_{t\geq 0}\cD_{t}$ is bounded.
Thus, for all~$k\in\N_0$ and all~$\hat{\zeta}\in\cD_{t_k}$ every function
$\zeta\in\cY_{\hat{\zeta}}([t_k,\infty))$ is bounded.
Since $O\times{\dB}_{\xi}$ is bounded,  it follows from the BIBS condition~\eqref{eq:BIBO-ID} that the set 
$N:=\bigcup_{k\in\N_0}\bigcup_{t\geq t_{k}} N^{t}_{t_k}$ is also bounded. Then the set $K:=\overline{O\times N}$ is compact and
\[
    \fa T>0\fa k\in\N_0\fa (\hat{\zeta},\hat{\eta})\!\in\! \cD_{t_k}\times
        \bigcup_{i=0}^{k}N_{t_i}^{t_k}
        \fa \zeta\in\cY_{\hat{\zeta}}([t_k,t_k+T]) \fa t\in\![t_k,t_k+T]\!:
        (\zeta(t),\eta(t;t_k,\hat{\eta},\zeta))\!\in\! K.
\]
To see this, let $T>0$, $k\in\N_0$, and  $(\hat{\zeta},\hat{\eta})\in \cD_{t_k}\times \bigcup_{i=0}^{k}N_{t_i}^{t_k}$ be arbitrarily given. 
Then, there exists $i\leq k$ with $\hat{\eta}\in N_{t_i}^{t_k}$. By definition of $N_{t_i}^{t_k}$, there exist
$(\hat{\zeta}_0,\hat{\eta}_0)\in\cD_{t_i}\times{\dB}_{\xi}$ and $\zeta_0\in\cY_{\hat{\zeta}_0}([t_i,\infty))$ such that
$\hat{\eta}=\eta(t_k;t_i,\hat{\eta}_0,\zeta_0)$. Let $\zeta\in\cY_{\hat{\zeta}}([t_k,t_k+T])$, then 
$\zeta(t)\in\cD_t\subseteq O$ for all $t\in[ t_k,t_k+T]$.
Define $\tilde{\zeta}:[t_i,\infty)\to \R^{m}$ by 
\[
    \tilde{\zeta}(t) :=
    \begin{cases}
       \zeta(t),&t\in[t_k,t_k+T],\\ 
       \zeta_0(t),&t\in[t_i,t_k)\cup(t_k+T,\infty).
    \end{cases}
\]
Then, $\tilde{\zeta}\in\cY_{\hat{\zeta}_0}([t_i,\infty))$ and for $t\in[ t_k,t_k+T]$ we have
\[
    \eta(t;t_k,\hat{\eta},\zeta)=\eta(t;t_k,\eta(t_k;t_i,\hat{\eta}_0,\zeta_0),{\zeta})=\eta(t;t_i,\hat{\eta}_0,\tilde{\zeta})\in N_{t_k}^t \subseteq N.
\]

\noindent
\emph{Step four.} 
A straightforward adaption of~\cite[Prop.~4.9]{BergDenn21}, using the constructed compact set~$K$,
yields the existence of~$M>0$ such that for all~$k\in\N_0$ the set~$\cU(t_k,\hat{x})$ 
is non-empty if $\hat{x}\in\Phi^{-1}\rbl
        \cD_{t_k}\times
        \bigcup_{i=0}^{k}N_{t_i}^{t_k}
        \rbr$.

\noindent
\emph{Step five.}
We show by induction that
\[
    \fa k\in\N_0:\ x_k^{\rm pre}\in \Phi^{-1}
        \rbl
        \cD_{t_k}\times
        \bigcup_{i=0}^{k}N_{t_i}^{t_k}
        \rbr
\]
and  
\[
    \fa k\in\N_0 \fa t\in [t_k,t_{k+1}]:\quad h(x(t;t_k,\kappa_{\xi}(x_k^{\rm pre},\hat{y}_{k}),u^\star_{k}))\in\cD_t.
\]
Since~$x_0^{\rm pre}\in 
\setdef{x\in\R^n}
{
    h(x)\in\cD_{t_0},\ \Norm{[0,I_{n-m}]\Phi(x)} \le \xi
}$ by assumption, 
$ x_0^{\rm pre}\!\in 
  \Phi^{-1}
        \rbl
        \cD_{t_{0}}\times
        N_{t_{0}}^{t_{0}}
        \rbr$.
Due to observation~\ref{Item:PredictionInitVal} of~{Step two}, $\kappa_{\xi}(x_0^{\rm pre},\hat{y}_{0})\in 
  \Phi^{-1}
        \rbl
        \cD_{t_{0}}\times
        N_{t_0}^{t_{0}}
        \rbr$. 
Thus,  $\cU(t_0,\kappa_{\xi}(x_0^{\rm pre},\hat{y}_{0}))\neq\emptyset$ according to~{Step four}.
The optimization problem~\eqref{eq:Prop:OPC} has a solution $u^\star_0\in \cU(t_0,\kappa_{\xi}(x_0^{\rm pre},\hat{y}_{0}))$
because of observation~\ref{Item:SolutionOptimizationProb} in  Step two. Due to the definition of
$\cU(t_0,\kappa_{\xi}(x_0^{\rm pre},\hat{y}_{0}))$, this implies in particular
$h(x(t;t_0,\kappa_{\xi}(x_0^{\rm pre},\hat{y}_{0}),u^\star_{0}))\in\cD_t$ for all $t\in [t_0,t_1]$ 
and
\[
x_1^{\rm pre}=x(t_1;t_0,\kappa_{\xi}(x_0^{\rm pre},\hat{y}_{0}),u^\star_{0})\in
\Phi^{-1}
\rbl
    \cD_{t_{1}}\times
    \bigcup_{i=0}^{1}N_{t_i}^{t_{1}}
\rbr
\]
according to observation~\ref{Item:RecInitialVal} in Step two.\\
If $x_k^{\rm pre}\in \Phi^{-1}
            \rbl
            \cD_{t_k}\times
            \bigcup_{i=0}^{k}N_{t_i}^{t_k}
            \rbr$ for $k\in\N$, then $\kappa_{\xi}(x_k^{\rm pre},\hat{y}_{k})\in\Phi^{-1}
            \rbl 
            \cD_{t_k}\times
            \bigcup_{i=0}^{k}N_{t_i}^{t_k}
            \rbr$
due to observation~\ref{Item:PredictionInitVal} in Step two.
Thus,  $\cU(t_k,\kappa_{\xi}(x_k^{\rm pre},\hat{y}_{k}))\neq\emptyset$ according to Step four.
Because of observation~\ref{Item:SolutionOptimizationProb} in Step two,
the OCP~\eqref{eq:Prop:OPC} has a solution $u^\star_k\in \cU(t_k,\kappa_{\xi}(x_k^{\rm pre},\hat{y}_{k}))$.
By definition of
$\cU(t_k,\kappa_{\xi}(x_k^{\rm pre},\hat{y}_{k}))$, this results in 
$h(x(t;t_k,\kappa_{\xi}(x_k^{\rm pre},\hat{y}_{k}),u^\star_{k}))\in\cD_t$ for all $t\in [t_k,t_{k+1}]$ 
and 
\[
x_{k+1}^{\rm pre}=x(t_{k+1};t_k,\kappa_{\xi}(x_k^{\rm pre},\hat{y}_{k}),u^\star_{k})\in
\Phi^{-1}
\rbl
    \cD_{t_{k+1}}\times
    \bigcup_{i=0}^{k+1}N_{t_i}^{t_{k+1}}
\rbr
\]
according to observation~\ref{Item:RecInitialVal} in Step two.

\noindent
\emph{Step six.}
It follows from~{Step five} that $y_{\rm M}(t)\in\cD_t$ for all $t\in\Rp$.
Define $y_{{\rm M},k}:=y_{{\rm M}}|_{[t_k,t_{k+1}]}$
and $\hat{\eta}_k:=[0,I_{n-m}]\Phi(\kappa_{\xi}(x_k^{\rm pre},\hat{y}_{k}))$ for all $k\in\N_0$.
Due to the definition of the compact set~$K$ and since $y_{{\rm M},k}\in\cY_{y_{{\rm M},k}(t_k)}([t_k,t_{k+1}])$,
we have $(y_{{\rm M},k}(t),\eta(t;t_k,\eta_k,y_{{\rm M},k}))\in K$ for all $t\in[t_k,t_{k+1}]$ and all $k\in\N_0$.
The functions $y_{{\rm M},k}$ and $\eta(\cdot;t_k,\eta_k,y_{{\rm M},k})$ satisfy the differential equation~\eqref{eq:BIF}
on the interval~$[t_k,t_{k+1}]$ for all $k\in\N$, thus
\begin{align*}
    \esssup_{t \ge 0}\Norm{\dot{y}_{\rm M}(t)}
    &\leq \sup_{k\in\N_0}\esssup_{t \in[t_k,t_{k+1}]}\Norm{\dot{y}_{{\rm M},k}(t)}\\
    &= \sup_{k\in\N_0}\esssup_{t \in[t_k,t_{k+1}]}\Norm{p(y_{{\rm M},k}(t), \eta(t;t_k,\hat{\eta}_k,y_{{\rm M},k}))
        +\Gamma(\Phi^{-1}(y_{{\rm M},k}(t), \eta(t;t_k,\hat{\eta}_k,y_{{\rm M},k})) )u^\star_k(t) }\\
    &\leq \sup_{k\in\N_0}\esssup_{t \in[t_k,t_{k+1}]}\Norm{p(y_{{\rm M},k}(t), \eta(t;t_k,\hat{\eta}_k,y_{{\rm M},k}))}
        +\Norm{\Gamma(\Phi^{-1}(y_{{\rm M},k}(t), \eta(t;t_k,\hat{\eta}_k,y_{{\rm M},k})) )}\!\SNorm{u_k^\star }\\
    &\leq \max_{(y, \eta) \in K} \Norm{ p(y, \eta) }  + M \max_{(y, \eta) \in K} \Norm{\Gamma(\Phi^{-1}(y, \eta) )} =:\bar{\lambda}.
   \end{align*}
The last inequality holds 
for all choices of~$(\hat{y}_k)_{k\in\N_0}$, $\delta$, $x_0^{\rm pre}$, and~$\kappa_{\xi}$, which completes the proof.
\end{proof}

\subsection{Auxiliary funnel control results} \label{Sec:FCResults}
Before proving~\Cref{Prop:FC_dist}, concerning the application of funnel control, we state the following result.
\begin{lemma}\label{lemma:ActivationSurjective}
    Let $N\in\con(\Rp,\R)$ be a surjection, $\alpha\in\con([0,1),[1,\infty))$ be a bijection,
    and $\beta\in\con([0,1],[0,\beta^+])$ be an activation function with $\beta^+>0$.
    Then $\tilde{N}:=(\beta\circ\sqrt{\alpha^{-1}})\cdot N\in\con(\Rp,\R)$ is surjective.
\end{lemma}
\begin{proof}
$N\in\con(\Rp,\R)$ being a surjection is equivalent to 
$\limsup_{s\to \infty} N(s)=\infty$ and $\liminf_{s\to \infty} N(s)=-\infty$.
Since $\lim_{s\to \infty} (\beta\circ\sqrt{\alpha^{-1}})(s)=\beta^+>0$, we have
\[
    \limsup_{s\to \infty} \tilde{N}(s)=\infty\quad \text{ and }\quad \liminf_{s\to \infty} \tilde{N}(s)=-\infty.
\]
This implies that $\tilde{N}=(\beta\circ\sqrt{\alpha^{-1}})\cdot N\in\con(\Rp,\R)$ is surjective as well.
\end{proof}
\begin{proof}[\textbf{Proof of~\Cref{Prop:FC_dist}}]
    Using~\cref{lemma:ActivationSurjective} and the {perturbation} high-gain property from~\Cref{Def:system-class}~\ref{Item:high-gain-prop} it is a {straightforward modification of the proof of~\cite[Thm.~1.9]{BergIlch21}, using the perturbation high-gain property instead of the high-gain property, similar as in Step four of the proof of~\Cref{Thm:RFMPC}.}
\end{proof}

\subsection{Proof of the main result} \label{Sec:ProofMain}
Now we are in the position to present the proof of the main result \Cref{Thm:RFMPC}.
\begin{proof}[\textbf{Proof of~\Cref{Thm:RFMPC}}]
Let $\Phi$ be a diffeomorphism associated with the model $(f,g,h)$ according to \Cref{ass:BIF}.

\noindent
\emph{Step one.} We show that the set $X^0$ of initial values for the model is non-empty.
    To this end, let $z:=\Phi^{-1}(y^0(0),0_{n-m})$. Then, recalling $[I_m,0]\Phi(\cdot) = h(\cdot)$,
    we have $h(z)=y^0(0)$. 
    Therefore,  $\Norm{h(z) - y_{\rf}(0)}=\Norm{y^0(0)-y_{\rf}(0)}<\psi(0) $
    because $y^0(0)\in\cD_0$.
    Further, $ \Norm{y(0)-h(z)}= 0$ and  $\Norm{[0,I_{n-m}]\Phi(z)}=\Norm{[0,I_{n-m}]\Phi(\Phi^{-1}(y^0(0),0_{n-m}))} = 0 \leq\xi $.\
    Thus, $z\in X^0$.
    
\noindent
\emph{Step two.}
According to \Cref{Prop:OCP_has_solution} there exists~$M>0$ such that, for every 
$\hat x_0 \in X^0$ and every possible sequence of measurements~$(\hat{y}_k)_{k\in\N_0}$ with $\hat y_k\in\cD_{t_k}$ for all $k\in\N_0$,
the OCP~\eqref{eq:RobustFMPCOCP}  has a solution~$u_{k,\rm FMPC}$ for every~$k\in\N_0$ and initialization
$\hat x_k = \kappa_{\xi}(x^{\rm pre}_{k},\hat y_k)$ of the model~\eqref{eq:Mod}, where  $x^{\rm pre}_{k+1} = x(t_{k+1};t_k,\hat x_k, u_{k,\rm FMPC})$.

\noindent
\emph{Step three.} 
Now we turn towards the part where the funnel controller~\eqref{alg:eq:FC} is involved.
On each interval $[t_k,t_{k+1}]$ the system's dynamics are given by
\begin{equation} \label{eq:dot_yk}
    \dot y_k(t) = F(d(t),\oT(y_k)(t),u_k(t)), \quad y_k|_{[-\sigma,t_k]} = y_{k-1}|_{[-\sigma,t_k]},
\end{equation}
where $y_{-1}|_{[-\sigma,0]} := y^0$, and in particular $y_k(t_k) = y_{k-1}(t_k)$, i.e., although the model's state is updated at $t = t_k$, the system is not re-initialized at the time instances $t_k$.
The funnel control signal is, for $k \in \N_0$ and $t \in [t_k,t_{k+1}]$, given by
\begin{equation*} 
    u_{k,\rm FC}(t) =  \beta(\|y_k(t) - y_{k,\rm M}(t)\|/\vp_k(t)) (N \circ \alpha)(\|  (y_k(t) - y_{k,\rm M}(t))/\vp_k(t) \|^2)  (y_k(t) - y_{k,\rm M}(t))/\vp_k(t) , 
\end{equation*}
where 
$y_{k,\rm M}(t) = h(x(t;t_k,\kappa_\xi(x_k^{\rm pre},y_k(t_k)),u_{k,\rm FMPC}))$ for $t \in [t_k,t_{k+1}]$
and the funnel function for the funnel control law is piecewise defined by
\begin{align*}
    \vp_k : [t_k,t_{k+1}] &\to \R, \quad
    t  \mapsto {\psi(t) - \|y_{k,\rm M}(t) - y_{\rm ref}(t)\|}, \quad k \in \N_0.
\end{align*}
Invoking \Cref{Prop:OCP_has_solution}, we have 
$y_{k,\rm M}(t) \in \cD_t$ for all~$t\in [t_k,t_k+T]$, 
by which $\vp_k$ satisfy $0 < \vp_k(t) \le \psi(t)$ for all $t \in [t_k,t_{k+1}]$ and all $k \in \N_0$.
Every $\vp_k$ can smoothly be extended to the left and right such that the extension $\tilde \vp_k$ satisfies $\tilde \vp_k \in \cG$ for all $k \in \N_0$.
We show that the control law
\begin{equation}\label{eq:uk}
u_k(t) = u_{k,\rm FMPC}(t) +  u_{k,\rm FC}(t)
\end{equation}
applied to the system~\eqref{eq:dot_yk} for $k\in\N_0$, leads to a closed-loop system which has a global solution with the properties as in \Cref{Prop:FC_dist}.
Special attention is required since $y_{k,\rm M}(t_k) \neq y_{k-1,\rm M}(t_k)$ and hence also $\vp_k(t_k)\neq \vp_{k-1}(t_k)$ is possible.
We observe that for $x^0 \in X^0$ we have 
\begin{align*}
&\| y_{0,\rm M}(0) - y_0(0)\| = \| h(x^0) - y_0(0) \|  < \psi(0) - \| h(x^0) - y_{\rm ref}(0) \| = {\vp_0(0)}
\end{align*}
and $y_0(0) \in \cD_0$.
Then \Cref{Prop:OCP_has_solution} yields $ \| u_{0,\rm FMPC} \|_\infty \le M$. 
Thus,
the feasibility result \Cref{Prop:FC_dist} for the funnel controller is applicable and yields the existence of a solution $y_0 : [0,t_1] \to \R^m$ of the closed-loop problem~\eqref{eq:dot_yk},~\eqref{eq:uk} for $k=0$, with $ \| y_{0,\rm M}(t) - y_0(t)\| < \vp_0(t)$ for all $t\in[t_0,t_1]$.
Then, choosing $\hat x_1 = \kappa_\xi(x_1^{\rm pre},y_{0}(t_1))\in \Omega_{\xi}(x_1^{\rm pre},y_{0}(t_1)) $, at $t=t_1$ 
we have either $ \hat x_1 = x_1^{\rm pre}$, which gives $y_{1,\rm M}(t_1) = y_{0,\rm M}(t_1)$ and thus
\begin{equation*}
\| y_{1,\rm M}(t_1) - y_1(t_1) \| = \| y_{0,\rm M}(t_1) - y_0(t_1) \| < \vp_0(t_1) =\psi(t_1) - \| y_{0,\rm M}(t_1) - y_{\rm ref}(t_1) \| = \vp_1(t_1),
\end{equation*}
or $y_{1,\rm M}(t_1) = h(\hat{x}_1) = y_0(t_1) = y_1(t_1)$ and the estimation above is valid as well, thus $y_1(t_1) \in \cD_{t_1}$. \Cref{Prop:OCP_has_solution} yields $\| y_{1,\rm M}(t) - y _{\rm ref}(t)\| < \psi(t)$ for $t \in [t_1,t_2]$ with $\|u_{1,\rm FMPC}\|_\infty \le M$, by which the conditions to reapply \Cref{Prop:FC_dist} are satisfied at $t = t_1$, by which a solution $y_1 : [t_1,t_2] \to \R^m$ of the closed-loop problem~\eqref{eq:dot_yk},~\eqref{eq:uk} exists for $k=1$, with $ \| y_{1,\rm M}(t) - y_1(t)\| < \vp_1(t)$ for all $t\in[t_1,t_2]$. Repeating this line of arguments we successively obtain, for each $k\in\N_0$, a solution $y_k : [t_k,t_{k+1}] \to \R^m$ of the closed-loop problem~\eqref{eq:dot_yk},~\eqref{eq:uk} with $ \| y_{k,\rm M}(t) - y_k(t)\| < \vp_k(t)$ for all $t\in[t_k,t_{k+1}]$.

\noindent
\emph{Step four.} 
By defining $y:[-\sigma,\infty)\to\R^m$ via $y|_{[-\sigma,0]} = y^0$, $y|_{[t_k,t_{k+1}]} = y_k$ for $k\in\N_0$ we obtain a global solution of~\eqref{eq:Sys}, \eqref{eq:u} which satisfies 
$    \|y_{\rm M}(t) - y(t)\| < \psi(t) - \|y_{\rm M}(t) - y_{\rf}(t)\| = \vp(t)$
 for all $t\ge 0$, where $y_{\rm M}(t) := y_{k,\rm M}(t)$ and $\vp(t):= \vp_k(t)$ for $t\in[t_k,t_{k+1})$, $k\in\N_0$. It remains to show that the overall control $u(t) := u_k(t),\qquad t\in[t_k,t_{k+1}),\quad k\in\N_0$, is bounded, which we prove by  showing that there exists $\ve \in(0, 1)$ 
such that $\| y(t) - y_{\rm M}(t) \| \le \ve \vp(t)$ for all $t \ge 0$. 
For the sake of better legibility, we introduce the variable $w(t) :=  (y(t) - y_{\rm M}(t))/\vp(t)$. 
Choose compact sets $K_p\subset\R^p$ and $K_q\subset\R^q$ such that $d(t) \in K_p$ and $\oT(y)(t) \in K_q$ for $t \ge 0$.
Further, for $K_{m} := \setdef{D\in\R^m}{\|D\|\le M}$, $\nu \in (0,1)$ and $V := \setdef{ v \in \R^m}{ \nu \le \|v\| \le 1}$ 
we recall the continuous function from~\Cref{Def:system-class}~\ref{Item:high-gain-prop}
\begin{equation*}
    \chi(s) = \min \setdef{ \langle v, F(\delta,\zeta,\Delta -s v) \rangle }{ \delta \in K_{p}, \zeta \in K_q,\Delta\in K_m, v \in V }.
\end{equation*}
$F$ has the {perturbation} high-gain property and hence the function $\chi$ is unbounded from above for a suitable $\nu\in(0,1)$.
We note that $\|w(0)\|<1$ as shown in Step~three, 
and with $\lambda : = \| \dot \psi \|_\infty + \| \dot y_{\rm ref}\|_\infty$ and $\bar \lambda \ge \|\dot y_{\rm M}\|_\infty$ from \Cref{Prop:OCP_has_solution},
we {choose $\ve\in(0,1)$ large enough such that $\ve > \max\{\nu, \|w(0)\|\}$ and
$
    \chi( \beta(\ve) (N\circ \alpha)(\ve^2) )  \ge 4 \bar \lambda + 2 \lambda,
$
which is possible because of the} properties of $\beta,N,\alpha$ and $\Tilde{\chi}$.
We show that $\|w(t)\| \le \ve$ for all~$t \ge 0$. 
Unlike the standard funnel control framework, the funnel function $\vp$ may have discontinuities at the time instances $t_k$ when the model is re-initialized with
$\hat x_k \in \Omega_{\xi}(x_{k}^{\rm pre}, y(t_k))$ such that $h(\hat x_k) = y(t_k)$.
This fact requires particular attention when proving $\|w(t)\| \le \ve$ for all~$t \ge 0$.
We observe that $\vp$ is continuous on $[t_k,t_{k+1}]$ for all $k\in\N_0$ and satisfies, by \Cref{Prop:OCP_has_solution},
\begin{equation*}
 \begin{aligned}
   | \dot \vp(t) | &  \le | \dot \psi(t) | + \| \dot y_{\rm M}\| + \| \dot y_{\rm ref}(t) \| 
   \le \lambda + \bar \lambda
\end{aligned}  
\end{equation*}
for almost all $t \ge 0$, independent of $k$. Now fix an arbitrary $k \in \N_0$ and consider two cases. \\
\textit{Case 1} : If $\hat x_k \in \Omega_{\xi}(x_{k}^{\rm pre}, y_{k-1}(t_k))$ is such that $h(\hat x_k) = y_{k-1}(t_k)$, then {$y_{\rm M}(t_k) = y(t_k)$ and hence} $\|w(t_k) \| = 0 < \ve$.
Seeking a contradiction, we suppose that there exists $t^* \in {(t_k, t_{k+1}]}$ such that {$\| w(t^*)\| > \ve$,} and invoking continuity of $w$ on $[t_k,t_{k+1}]$ we set $t_* := \sup\setdef{ t\in [t_k,t^*)}{ \|w(t)\| = \ve} < t^*$.
Then {we have $\|w(t)\|\ge \ve \ge \nu$ (and hence $w(t)\in V$) for all $t \in [t_*,t^*]$ and, since $\|w(t_*)\| = \ve$,
$ \chi(\beta(\|w(t_*)\| )(N \circ \alpha)(\|w(t_*)\|^2)) \ge 4 \bar \lambda + 2 \lambda. $
Therefore, there exists $t^{**}\in (t_*,t^*]$ such that
\[ 
   \forall\, t\in [t_*,t^{**}]:\   \chi(\beta(\|w(t)\| )(N \circ \alpha)(\|w(t)\|^2)) \ge 2 \bar \lambda + \lambda.
\]
Then we calculate that, for almost all $t \in [t_*,t^{**}]$,}
\begin{small}
\begin{align*}
    \ddt \tfrac{1}{2} \| w(t)\|^2 &= \langle w(t), \dot w(t) \rangle 
     = \left\langle w(t), \frac{-\dot \vp(t) (y(t) - y_{\rm M}(t)) + \vp(t) (\dot y(t) - \dot y_{\rm M}(t))}{\vp(t)^2} \right\rangle \\
& = - \frac{\dot \vp(t)}{\vp(t)} \langle w(t) , w(t) \rangle - \frac{1}{\vp(t)} \langle w(t), \dot y_{\rm M}(t) \rangle + \frac{1}{\vp(t)} \langle w(t) , F(d(t), \oT(y)(t), u(t)) \rangle \\
& < \frac{1}{\vp(t)} \Big( | \dot \vp(t) | + \|\dot y_{\rm M}(t)\| + \langle  w(t) , F(d(t), \oT(y)(t), u(t)) \rangle \Big) \\
&\le \frac{1}{\vp(t)} (\lambda + 2 \bar \lambda)
+
 \frac{1}{\vp(t)} \langle  w(t) , F\big(d(t), \oT(y)(t), u_{k,\rm FMPC}(t) + u_{k,\rm FC}(t)\big)\rangle  \\
& \le \frac{1}{\vp(t)} (\lambda + 2 \bar \lambda)
- \frac{1}{\vp(t)} \min \setdef{ \!\! \langle v, F(\delta,\zeta, \!\Delta-\! \beta(\|w(t)\|) (N \! \circ \! \alpha)(\|w(t)\|^2) v) \rangle \!\! }{  \! \delta \!\in\! K_{{p}}, \zeta \!\in\! K_q, \Delta\!\in\! K_m, v \!\in\! V \!\!} \\
& \le \frac{1}{\vp(t)} \Big( \lambda + 2 \bar \lambda -  \chi(\beta(\|w(t)\| )(N \circ \alpha)(\|w(t)\|^2)) \Big) 
 \le 0,
\end{align*}%
\end{small}%
where we used $u_{, \rm FC}(t) = \beta(\|w(t)\|) (N\circ \alpha)(\|w(t)\|^2) w(t)$ in the penultimate inequality.
Upon integration, and invoking the definition of $t^* < t^{**}$, this gives
$
    \ve < \| w(t^{**})\| \le \| w(t_*)\| = \ve
$,
a contradiction.
Therefore, $\| w(t) \| \le \ve$ for all $t \in [t_k, t_{k+1}]$. \\
\textit{Case 2}: If $\hat x_k = x_k^{\rm pre}$, then $\vp_{k-1}(t_k) = \vp_k(t_k)$ and thus the funnel function $\vp$ is continuous and weakly differentiable on the interval $[t_{k-1},t_{k+1}]$.
In this case, it follows that $\|w(t) \| \le \ve$ for all $t \in [t_{k-1}, t_{k+1}]$ with the same arguments as in Case~1.\\
Overall, we have shown that $\|w(t)\| \le \ve$ for all $t \in [t_k,t_{k+1}] $ and all $k \in \N_0$, independent of the initialization strategy. Therefore,
$\| u \|_\infty \le M + \beta^+ |(N \circ \alpha)(\ve^2)|$ and this proves assertion~\ref{Assertion:y_u_bounded}.

\noindent
\emph{Step five.} Finally, a simple calculation yields that for $ t \ge 0$ we have
\begin{align*}
 \ \| y(t) - y_{\rm ref}(t) \| 
 &= \| y(t) - y_{\rm M}(t) + y_{\rm M}(t) - y_{\rm ref}(t) \| 
  \le \| y(t) - y_{\rm M}(t) \| + \| y_{\rm M}(t) - y_{\rm ref}(t) \| \\
 & < \vp(t) + \| y_{\rm M}(t) - y_{\rm ref}(t) \| 
  = \psi(t) - \| y_{\rm M}(t) - y_{\rm ref}(t) \| + \| y_{\rm M}(t) - y_{\rm ref}(t) \| = \psi(t),
\end{align*}
which is assertion~\ref{Assertion:tracking_error}. 
This completes the proof.
\end{proof}

\bibliographystyle{plain}
\footnotesize
\bibliography{\References}

\end{document}